\documentclass[reqno]{amsart}
\usepackage{amsmath,amsthm,amssymb,amsfonts,mathrsfs,amscd}
\usepackage{mathtools}
\usepackage{varwidth}
\usepackage{hyperref}
\usepackage{multirow}
\usepackage{graphicx}
\usepackage{bm}
\usepackage{mathrsfs}
\usepackage{dsfont}
\usepackage{tikz}
\usepackage{tikz-3dplot}
\usepackage{subfig}
\usepackage{stmaryrd}
\usepackage{makecell}
\usepackage{booktabs}

\usepackage{float}
\usepackage[textfont=it]{caption}

\DeclareMathOperator\supp{supp}
\renewcommand{\theequation}{\arabic{section}.\arabic{equation}}

\newcommand*{\Vector}[1]{{\boldsymbol{\mathbf{#1}}}}
\newcommand*{\Diff}{\mathop{}\!\mathrm{d}}
\newcommand*{\imp}{\mathrm{imp}}
\newcommand*{\Order}{\operatorname{O}}

\theoremstyle{plain}
\newtheorem{theorem}{Theorem}[section]
\newtheorem{lemma}[theorem]{Lemma}
\newtheorem{corollary}[theorem]{Corollary}
\newtheorem{remark}[theorem]{Remark}

\newtheorem{assumption}{Assumption}[section]
\newtheorem{definition}{Definition}[section]

\newcommand{\vertiii}{{\vert\kern-0.25ex\vert\kern-0.25ex\vert}}

\newtheoremstyle{claim}
  {}
  {}
  {\it}
  {}
  {\rm}
  {.}
  {.5em}
  {}
\theoremstyle{claim}

\makeatletter
  \@addtoreset{equation}{section}
  \newcommand\figcaption{\def\@captype{figure}\caption}
  \newcommand\tabcaption{\def\@captype{table}\caption}
\makeatletter

\begin{document}

\title{A hybrid two-level weighted Schwartz method for time-harmonic Maxwell equations}

\author{Ziyi Li}
\author{Qiya Hu*}

 \thanks{1. LSEC, ICMSEC, Academy of Mathematics and Systems Science, Chinese Academy of Sciences, Beijing
 100190, China; 2. School of Mathematical Sciences, University of Chinese Academy of Sciences, Beijing 100049,
 China (liziyi@lsec.cc.ac.cn, hqy@lsec.cc.ac.cn). The work of the author was supported by the National Natural Science Foundation of China grant G12071469.\\
 $\ast$ Correspondence author}

\maketitle

{\bf Abstract.} This paper concerns the preconditioning technique for discrete systems arising from time-harmonic Maxwell equations with absorptions, where the discrete systems are generated by N\'ed\'elec finite
element methods of fixed order on meshes with suitable size. This kind of preconditioner is defined as a two-level hybrid form, which falls into the class of ``unsymmetrically weighted'' Schwarz method based on the overlapping domain decomposition with impedance boundary subproblems. The coarse space in this preconditioner is constructed by specific eigenfunctions solving a series of generalized eigenvalue problems in the local discrete Maxwell-harmonic spaces according to a user-defined tolerance $\rho$. We establish a stability result for the considered discrete variational problem. Using this discrete stability, we prove that the two-level hybrid Schwarz preconditioner is robust in the sense that the convergence rate of GMRES is independent of the mesh size, the subdomain size and the wave-number when $\rho$ is chosen appropriately. We also define an economical variant that avoids solving generalized eigenvalue problems. Numerical experiments confirm the theoretical results and illustrate the efficiency of the preconditioners.



\pagestyle{myheadings}
\thispagestyle{plain}
\markboth{}{Ziyi Li and Qiya Hu}

\section{Introduction}

Let $\Omega$ be a bounded, connected and Lipschitz domain in $\mathbb{R}^3$. Consider the following time-harmonic Maxwell equation with absorption:
\begin{equation}\label{eq:shiftMaxwell}
    \left\lbrace
    \begin{aligned}
         & \nabla\times(\nabla\times\Vector{E}) -(\kappa^2+\mathrm{i}\epsilon)\Vector{E}=\Vector{J}, & & \text{in }\Omega, \\
         & (\nabla\times\Vector{E})\times\Vector{n} -\mathrm{i}\kappa\Vector{E}_T=\Vector{g}, & & \text{on }\partial\Omega,
    \end{aligned} \right.
\end{equation}
where ${\bf n}$ denotes the unit outward normal on the boundary $\partial\Omega$, $\Vector{E}_T =(\Vector{n}\times\Vector{E})\times\Vector{n}|_{\partial\Omega}$ is the tangential component of $\Vector{E}$
on the boundary $\partial\Omega$, $\kappa$ is the wave number and $\epsilon$ is an absorption parameter.

As the fundamental model describing the propagation and scattering of electromagnetic waves, time-harmonic Maxwell equations have wide applications in engineering fields. Numerical simulation of this kind of equation
is undoubtedly an important topic. The discrete systems arising from time-harmonic Maxwell equations \eqref{eq:shiftMaxwell} with a large wave number $\kappa$ are often huge because of the so-called ``wave-number pollution" \cite{melenk2010convergence,melenk2023wavenumber, melenk2011wavenumber,wu2014pre}. Since the coefficient matrices of these systems are highly indefinite and heavily
ill-conditioned as $\kappa$ increases, the classical preconditioning techniques for $H(\bf{curl})$-elliptic problems \cite{hu2023convergence, hu2004substructuring, toselli2000overlapping} are inefficient for
time-harmonic Maxwell equations with a large wave number $\kappa$ and a small (or mild) absorption parameter $\epsilon$ ($\epsilon\ll\kappa^2$).
For example, the standard additive Schwarz method (ASM) for the edge finite element system arising from \eqref{eq:shiftMaxwell} with $\epsilon=0$ converges only if the coarse-grid size $H\leq C\kappa^{-1}$ with a sufficiently small constant $C$ \cite{gopalakrishnan2003overlapping}. It was shown in \cite{bonazzoli2019domain} that the standard additive Schwarz preconditioner for the edge finite element system arising from \eqref{eq:shiftMaxwell} possesses uniform convergence under the assumptions that $\epsilon\sim \kappa^2$, $H=\Order(\kappa^{-1})$ and the overlap size $\delta\sim H$. It is clear that the condition $\epsilon\sim \kappa^2$ is too limited, and the requirement $H=\Order(\kappa^{-1})$ leads to a huge coarse problem to be solved when $\kappa$ is large. Designing an efficient solution method for \eqref{eq:shiftMaxwell} with a small absorption parameter $\epsilon$ (in particular $\epsilon=0$) is a difficult task. To our knowledge, the convergence theory for time-harmonic Maxwell equations with a large wave number $\kappa$ and a small (or mild) absorption parameters $\epsilon$ ($\epsilon \ll\kappa^2$) is still missing.

For the pure Helmholtz problem, one can introduce an absorption parameter $\epsilon>0$ and use the Helmholtz system with such an absorption parameter to precondition the pure Helmholtz system. Such an idea, which originated in \cite{erlangga2004class} and is called ``shifted Laplacian preconditioning'', brought new inspiration to the design and analysis of preconditioners for the pure Helmholtz problem. It was shown in
\cite {gander2015applying} that a shifted Helmholtz system is a robust preconditioner of the pure Helmholtz system when the ``shifted (absorption) parameter" $\epsilon$ satisfies $\epsilon\leq c_0\kappa$ with a proper constant $c_0$. According to the discussion in \cite{bonazzoli2019domain}, we expect that the above result can be extended to time-harmonic Maxwell equations, i.e., the discrete system arising from \eqref{eq:shiftMaxwell} with a slightly large absorption parameter $\epsilon$ ($\epsilon\sim \kappa$) is a good preconditioner of the one arising from \eqref{eq:shiftMaxwell} with small absorption parameters $\epsilon$ (in particular $\epsilon=0$). This inspires us that we need only to study the preconditioning technique for shifted  problem with relatively mild shifted (absorption) parameter $\epsilon$ (see \cite{graham2020domain} and \cite{bonazzoli2019domain} for the details).

In this paper we are concerned with an overlapping domain decomposition method for time-harmonic Maxwell equations \eqref{eq:shiftMaxwell} with relatively mild absorption parameters. The weighted additive Schwarz method
with local impedance conditions (WASI) shows higher efficiency in solving Helmholtz equations. The one-level WASI was analyzed in \cite{gong2023convergence, graham2020domain}, and the two-level WASI with spectral
coarse spaces was proposed and tested in \cite{bootland2021comparison, conen2014coarse}. Recently, a two-level hybrid WASI with a novel spectral space for Helmholtz equations was proposed and analyzed in \cite{hu2024novel}.
In the current paper, we extend the method proposed in \cite{hu2024novel} to the time-harmonic Maxwell equations \eqref{eq:shiftMaxwell} with mild absorption parameters $\epsilon$, which is discretized by the standard N\'ed\'elec finite element method. In this method, the coarse space is constructed by solving a series of local generalized eigenvalue problems in the local discrete Maxwell-harmonic spaces and by choosing specific eigenfunctions as local coarse basis according to a user-defined tolerance $\rho$. We focus on the convergence analysis of the proposed hybrid two-level Schwarz preconditioner for Maxwell discrete systems.
The key technical tool is a new stability estimate on the corresponding discrete variational problem of \eqref{eq:shiftMaxwell}. The stability of a discrete variational problem of \eqref{eq:shiftMaxwell} is vital
in the numerical analysis of \eqref{eq:shiftMaxwell}. Such stability may be derived by the coercivity or $inf$-$sup$ condition of the sesquilinear form on the considered discrete space.
There seem only a few works on the stability of the discrete variational problem of time-harmonic Maxwell equations. In \cite{bonazzoli2019domain} the stability of the discrete variational problem of \eqref{eq:shiftMaxwell} with $\epsilon\sim\kappa^2$ was directly established since the sesquilinear form is coercive on the edge finite element space when $\epsilon\sim\kappa^2$.
The paper \cite{feng2014absolutely} showed that, under suitable assumptions on the mesh size $h$ and the penalty parameters, the sesquilinear form associated with an interior penalty discontinuous Galerkin method for time-harmonic Maxwell equations is coercive on the considered discrete space, so the discrete variational problem is stable.

Unfortunately, for the considered situation with $\epsilon\ll\kappa^2$ the sesquilinear form derived by \eqref{eq:shiftMaxwell} is not coercive on the N\'ed\'elec finite element space, so we have to
build a discrete $inf$-$sup$ condition. Using a wavenumber-explicit estimate obtained in \cite{melenk2023wavenumber} and applying the technique developed in \cite{melenk2010convergence}, we establish
a rational $inf$-$sup$ condition of the sesquilinear form associated with \eqref{eq:shiftMaxwell} on the considered N\'ed\'elec finite element space.
With the help of this $inf$-$sup$ condition, we obtain a discrete stability result associated with \eqref{eq:shiftMaxwell} and further prove that the proposed preconditioner is robust in the sense that the convergence
rate of the preconditioned GMRES method with this preconditioner is independent of the mesh size $h$, the subdomain size $H$ and the wave-number $\kappa$, provided that the tolerance parameter $\rho$ is small enough.

The paper is organized as follows: In Section \ref{sec:preliminary}, we introduce variational problems of \eqref{eq:shiftMaxwell} and recall their basic properties. In Section \ref{sec:stability}, we establish several stability results for the variational problems with absorption. In Section \ref{sec:preconditioner}, we define a coarse space and construct a two-level overlapping domain decomposition preconditioner with the coarse solver defining
on the coarse space. In Section \ref{sec:convergence}, we give the convergence analysis of the proposed preconditioner for the case with a relatively mild $\epsilon$. In Section \ref{sec:economical}, we define an economical coarse space and the corresponding economical preconditioner. Finally, in Section \ref{sec:experiments} some numerical results are showed to illustrate the efficiency of the proposed preconditioners.

\section{Preliminary}\label{sec:preliminary}
\subsection{Notations}
For a domain $G$ in $\mathbb{R}^2$ or $\mathbb{R}^3$, we use $H^k(G)$ to denote the standard Sobolev space of order $k$. In particular, we have $H^0(G)=L^2(G)$.
Let $H^k(G)^n$ ($n=1,2,3,$) be the corresponding Sobolev space consisting of vector-valued functions with values in $\mathbb{C}^n$, and let $\Vert\cdot\Vert_{k,\Omega}$ denote the Sobolev norm on $H^k(G)^n$.
For $G\subset\mathbb{R}^3$, we could introduce the following vector-valued spaces:
\begin{align*}
    H(\Vector{curl};G)&\coloneqq \{\Vector{v}\in L^2(G)^3 :~ \nabla\times\Vector{v}\in L^2(G)^3\},\\
    L^2_T(\partial G) &\coloneqq \{\Vector{v}\in L^2(\partial G)^3:~ \Vector{v}\cdot\Vector{n} = 0\}.
\end{align*}
In this paper, we use $(\cdot,\cdot)_G$ and $\langle\cdot,\cdot\rangle_{\partial G}$ to denote the complex $L^2(G)^3$ inner product and $L^2_T(\partial G)$ inner product, respectively. For $G\subseteq \Omega$,
introduce the energy space
\begin{equation*}
    H_\imp(\Vector{curl};G)\coloneqq \{\Vector{v}\in H(\Vector{curl};G):~ \Vector{v}_T\in L^2_T(\partial G)\}.
\end{equation*}
We usually equip $H_\imp(\Vector{curl};G)$ with the $\kappa$-weighted inner product
\begin{equation*}
    (\Vector{v},\Vector{w})_{\imp,\kappa,G}\coloneqq (\nabla\times\Vector{v},\nabla\times\Vector{w})_G + \kappa^2 (\Vector{v},\Vector{w})_G + \kappa \langle\Vector{v}_T,\Vector{w}_T\rangle_{\partial G}
\end{equation*}
and the $\kappa$-weighted norm $\Vert\Vector{v}\Vert_{\imp,\kappa,G}\coloneqq (\Vector{v},\Vector{v})_{\imp,\kappa,G}^{\frac{1}{2}}$. For the convergence, we also use the following inner product
\begin{equation*}
    (\Vector{v},\Vector{w})_{\imp^-,\kappa,G}\coloneqq (\nabla\times\Vector{v},\nabla\times\Vector{w})_G + \kappa^2 (\Vector{v},\Vector{w})_G + \kappa \langle\Vector{v}_T,\Vector{w}_T\rangle_{\partial G\cap\Omega}
\end{equation*}
and the corresponding norm $\Vert\Vector{v}\Vert_{\imp^-,\kappa,G}\coloneqq (\Vector{v},\Vector{v})_{\imp^-,\kappa,G}^{\frac{1}{2}}$.

We will apply the finite element method to the discretization of (\ref{eq:shiftMaxwell}). Let $\mathcal{T}_h$ be a quasi-uniform tetrahedron grid posed on $\Omega$ with mesh size $h$, and $V_h(\Omega)$
be the standard N\'ed\'elec finite element space of fixed order $p$ built on $\mathcal{T}_h$ (see \cite{monk2003finite} for its precise definition). We also use $Z_h(\Omega)$ to denote the standard nodal finite element space of the same order built on $\mathcal{T}_h$. For a vector-valued function ${\bf v}$ with appropriate smoothness, we introduce its N\'ed\'elec interpolation ${\bf r}_h{\bf v}$ such that ${\bf r}_h{\bf v}\in V_h(\Omega)$ (see \cite{monk2003finite} for its precise definition). For a polyhedron $G\subseteqq\Omega$, let $V_h(G)$ (rep. $Z_h(G)$) denote the natural restriction of $V_h(\Omega)$ (rep. $Z_h(\Omega)$) on $G$.

Throughout this paper, we use the notation $a\lesssim b$ to represent that there exists a positive number $C$ such that $a\leq C b$, where the constant $C$ is independent of all parameters of interest (here $\epsilon,\kappa,h,H,\delta$ and $l$, with some of them defined later) if not specifically mentioned. Moreover, we say $a\sim b$ if $a\lesssim b$ and $b\lesssim a$. To conclude this subsection, we say domain $D\in\mathbb{R}^d$ has characteristic length $L$ if its diameter $\sim L$, the surface area $\sim L^{d-1}$ and its volume $\sim L^{d}$.

\subsection{Variational problem and its discretization}
We first make a basic assumption on $\kappa$ and $\epsilon$.
\begin{assumption}
    The parameters $\kappa$ and $\epsilon$ satisfy
    \begin{equation*}
        \kappa\gtrsim 1,\quad 0\leq\epsilon\leq\kappa^2.
    \end{equation*}
\end{assumption}

For a positive parameter $\epsilon>0$, define the sesquilinear form $a_{\epsilon}(\cdot,\cdot):~H_\imp(\Vector{curl};\Omega)\times H_\imp(\Vector{curl};\Omega)\rightarrow\mathbb{C}$ by
\begin{equation*}
    a_\epsilon(\Vector{E},\Vector{\xi})\coloneqq (\nabla\times\Vector{E},\nabla\times\Vector{\xi})_\Omega - (\kappa^2+\mathrm{i}\epsilon)(\Vector{E},\Vector{\xi })_{\Omega} - \mathrm{i}\kappa\langle\Vector{E}_T,\Vector{\xi}_T\rangle_{\partial\Omega}.
\end{equation*}
Let $\Vector{J}\in L^2(\Omega)^3$ and $\Vector{g}\in L^2_T(\partial\Omega)$, then the variational problem of \eqref{eq:shiftMaxwell} is: find $\Vector{E}\in H_\imp(\Vector{curl};\Omega)$ such that
\begin{equation}\label{eq:shiftMaxwellVF}
    a_\epsilon(\Vector{E},\Vector{\xi}) = \Vector{F}(\Vector{\xi}),\quad\forall\Vector{\xi}\in H_\imp(\Vector{curl};\Omega),
\end{equation}
where
\begin{equation*}
    \Vector{F}(\Vector{\xi})\coloneqq \int_\Omega\Vector{J}\cdot\overline{\Vector{\xi}}\Diff  V +\int_{\partial\Omega}\Vector{g}\cdot\overline{\Vector{\xi}}_T\Diff  S.
\end{equation*}
The discrete variational problem of \eqref{eq:shiftMaxwellVF} writes: find $\Vector{E}_h\in V_h(\Omega)$ such that
\begin{equation}\label{eq:discreteMVF}
    a_\epsilon(\Vector{E}_h,\Vector{\xi}_h) = \Vector{F}(\Vector{\xi}_h)\quad\forall \Vector{\xi}_h\in V_h(\Omega).
\end{equation}
Besides, define a discrete operator $A_\epsilon: V_h(\Omega)\rightarrow V_h(\Omega)$ as
\begin{equation*}
    (A_\epsilon \Vector{E}_h,\Vector{\xi}_h)_\Omega = a_\epsilon(\Vector{E}_h,\Vector{\xi}_h),\quad\forall \Vector{\xi}_h\in V_h(\Omega).
\end{equation*}
then \eqref{eq:discreteMVF} is equivalent to the operator equation
\begin{equation}\label{eq:operatorAhM}
    A_\epsilon \Vector{E}_h=\Vector{f}_h,\quad \Vector{E}_h, \Vector{f}_h\in V_h(\Omega),
\end{equation}
where $f_h$ is the Riesz representor of $\Vector{F}$ on $V_h(\Omega)$. It's well known that there exists the so-called ``wave number pollution" for the finite element discretization of the time-harmonic Maxwell equation. In order to eliminate this effect, we make the following assumption on the mesh size $h$.
\begin{assumption}\label{assump:meshsize}
    The mesh size $h$ satisfies  $h\sim \kappa^{-(1+\gamma)}$ for some $0<\gamma\leq 1$.
\end{assumption}

At the end, we introduce the adjoint sesquilinear form $\bar{a}_\epsilon(\cdot,\cdot)$ of $a_\epsilon(\cdot,\cdot)$,
\begin{equation*}
    \bar{a}_\epsilon(\Vector{E},\Vector{\xi})\coloneqq (\nabla\times\Vector{E},\nabla\times\Vector{\xi})_\Omega - (\kappa^2-\mathrm{i}\epsilon)(\Vector{E},\Vector{\xi })_{\Omega} + \mathrm{i}\kappa\langle\Vector{E}_T,\Vector{\xi}_T\rangle_{\partial\Omega}.
\end{equation*}
It's clear that $\bar{a}_\epsilon(\Vector{E},\Vector{\xi}) = \overline{a_\epsilon(\Vector{\xi},\Vector{E})}$.
\subsection{Basic properties of the sesquilinear forms}
As in the case of Helmholtz problems, the positive absorption makes the sesquilinear forms $a_\epsilon(\cdot,\cdot)$ and $\bar{a}_\epsilon(\cdot,\cdot)$ coercive.

\begin{lemma}[{Continuity and coercivity, \cite[Lemma 2.2, 2.5]{bonazzoli2019domain}}]\label{lemma:contCoerM}
    ~\\
    {\rm (1)} Let $0\leq\epsilon\leq\kappa^2$, then the sesquilinear forms $a_\epsilon(\cdot,\cdot)$ and $\bar{a}_\epsilon(\cdot,\cdot)$ are continuous:
    \begin{equation}\label{eq:continuousM}
        \max\{\vert a_\epsilon(\Vector{v},\Vector{w})\vert,\vert \bar{a}_\epsilon(\Vector{v},\Vector{w})\vert\} \lesssim \Vert\Vector{v}\Vert_{\imp,\kappa,\Omega} \Vert{\Vector{w}}\Vert_{\imp,\kappa,\Omega}
    \end{equation}
    for all $\Vector{v},\Vector{w}\in H_\imp(\Vector{curl};\Omega)$.
    \\
    {\rm (2)} Let $\epsilon>0$. The sesquilinear forms $a_{\epsilon}(\cdot,\cdot)$ and $\bar{a}_\epsilon(\cdot,\cdot)$ are coercive:
    \begin{equation}\label{eq:coerciveM}
      \min\{\vert a_\epsilon(\Vector{v},\Vector{v})\vert,\vert \bar{a}_\epsilon(\Vector{v},\Vector{v})\vert\} \gtrsim \frac{\epsilon}{\kappa^2} \Vert{\Vector{v}}\Vert_{\imp,\kappa,\Omega}^2
    \end{equation}
    for all $\Vector{v}\in H_\imp(\Vector{curl};\Omega)$.
\end{lemma}
From Lax-Milgram theorem and Lemma \ref{lemma:contCoerM}, the variational problems \eqref{eq:shiftMaxwellVF} and \eqref{eq:discreteMVF}(and the adjoint variational problems) are well-posed.

\section{Stability results for problems with absorption}\label{sec:stability}
In this section, we will build stability estimates of solutions to \eqref{eq:shiftMaxwellVF} and \eqref{eq:discreteMVF}, where $\Omega$ is replaced by a general domain $G\subseteqq\Omega$. The stability constants explicitly depend on
the wave number $\kappa$, the absorption parameter $\epsilon$ and the characteristic length of ${G}$. Stability estimates for Helmholtz equations with absorption and time-harmonic Maxwell equations without absorption were established in \cite{gander2015applying} and \cite{hiptmair2011stability}, respectively. In this section, we will follow the ideas in these works.

The first result is similar to \cite[Theorem 2.7]{gander2015applying}, which illustrates that the solution to \eqref{eq:shiftMaxwellVF} shows good stability when $\epsilon$ is relatively large.
\begin{theorem}\label{thm:maxwellGerStable}
    Let $\Vector{E}$ be the solution to \eqref{eq:shiftMaxwell}. Assume that $\epsilon\lesssim\kappa^2$, then there holds
    \begin{equation}\label{eq:maxwellGerStable}
        \Vert{\nabla\times\Vector{E}}\Vert_{0,{G}}^2 + \kappa^2\Vert{\Vector{E}}\Vert_{0,{G}}^2 + \kappa\Vert{\Vector{E}_T}\Vert_{0,\partial{G}}^2 \lesssim \frac{\kappa^2}{\epsilon^2}\Vert{\Vector{J}}\Vert_{0,{G}}^2 + \frac{\kappa}{\epsilon}\Vert{\Vector{g}}\Vert_{0,\partial{G}}^2.
    \end{equation}
\end{theorem}
\begin{proof}
    Taking $\Vector{\xi}=\Vector{E}$ in \eqref{eq:shiftMaxwellVF}, we have
    \begin{equation}\label{eq:proofGerStable1}
        (\kappa^2+\mathrm{i}\epsilon)\Vert{\Vector{E}}\Vert_{0,{G}}^2 - \Vert{\nabla\times\Vector{E}}\Vert_{0,{G}}^2 + \mathrm{i}\kappa\Vert{\Vector{E}_T}\Vert_{0,\partial{G}}^2 = -\int_{G} \Vector{J}\cdot\overline{\Vector{E}}\Diff V - \int_{\partial{G}}\Vector{g}\cdot\overline{\Vector{E}}_T\Diff S.
    \end{equation}
    Taking the imaginary part of \eqref{eq:proofGerStable1} and using Young's inequality give
    \begin{align*}
         \quad\epsilon \Vert{\Vector{E}}\Vert_{0,{G}}^2 + \kappa\Vert{\Vector{E}_T}\Vert_{0,\partial{G}}^2
         & \leq \Vert{\Vector{J}}\Vert_{0,{G}} \Vert{\Vector{E}}\Vert_{0,{G}} + \Vert{\Vector{g}}\Vert_{0,\partial{G}} \Vert{\Vector{E}_T}\Vert_{0,\partial{G}} \\
         & \leq \frac{1}{2\epsilon}\Vert{\Vector{J}}\Vert_{0,{G}}^2 + \frac{\epsilon}{2}\Vert{\Vector{E}}\Vert_{0,{G}}^2 + \frac{1}{2\kappa}\Vert{\Vector{g}}\Vert_{0,\partial{G}}^2 + \frac{\kappa}{2}\Vert{\Vector{E}_T}\Vert_{0,\partial{G}}^2,
    \end{align*}
    namely
    \begin{equation}\label{eq:proofGerStable2}
        \epsilon \Vert{\Vector{E}}\Vert_{0,{G}}^2 + \kappa\Vert{\Vector{E}_T}\Vert_{0,\partial{G}}^2 \leq \frac{1}{\epsilon}\Vert{\Vector{J}}\Vert_{0,{G}}^2 + \frac{1}{\kappa}\Vert{\Vector{g}}\Vert_{0,\partial{G}}^2.
    \end{equation}
    Then taking the real part and using Young's inequality again, we can obtain
    \begin{align*}
        \Vert{\nabla\times\Vector{E}}\Vert_{0,{G}}^2
         & \leq \kappa^2 \Vert{\Vector{E}}\Vert_{0,{G}}^2 + \Vert{\Vector{J}}\Vert_{0,{G}} \Vert{\Vector{E}}\Vert_{0,{G}} + \Vert{\Vector{g}}\Vert_{0,\partial{G}} \Vert{\Vector{E}_T}\Vert_{0,\partial{G}} \\
         & \leq \frac{3\kappa^2}{2}\Vert{\Vector{E}}\Vert_{0,{G}}^2 + \frac{\kappa}{2}\Vert{\Vector{E}_T}\Vert_{0,\partial{G}}^2 + \frac{1}{2\kappa^2}\Vert{\Vector{J}}\Vert_{0,{G}}^2  + \frac{1}{2\kappa}\Vert{\Vector{g}}\Vert_{0,\partial{G}}^2.
    \end{align*}
    Combining it with \eqref{eq:proofGerStable2} derives that
    \begin{align*}
         & \quad\Vert{\nabla\times\Vector{E}}\Vert_{0,{G}}^2 + \kappa^2\Vert{\Vector{E}}\Vert_{0,{G}}^2 + \kappa\Vert{\Vector{E}_T}\Vert_{0,\partial{G}}^2 \\
         & \leq \frac{5\kappa^2}{2}\Vert{\Vector{E}}\Vert_{0,{G}}^2 + \frac{3\kappa}{2}\Vert{\Vector{E}_T}\Vert_{0,\partial{G}}^2 + \frac{1}{2\kappa^2}\Vert{\Vector{J}}\Vert_{0,{G}}^2  + \frac{1}{2\kappa}\Vert{\Vector{g}}\Vert_{0,\partial{G}}^2 \\
         & \leq \left(\frac{5\kappa^2}{2\epsilon^2}+\frac{3}{2\epsilon}+\frac{1}{2\kappa^2}\right)\Vert{\Vector{J}}\Vert_{0,{G}}^2 + \left(\frac{5\kappa}{2\epsilon}+\frac{2}{\kappa}\right)\Vert{\Vector{g}}\Vert_{0,\partial{G}}^2 \\
         & \lesssim \frac{\kappa^2}{\epsilon^2}\left(1+\frac{\epsilon}{\kappa^2}+\left(\frac{\epsilon}{\kappa^2}\right)^2\right)\Vert{\Vector{J}}\Vert_{0,{G}}^2 + \frac{\kappa}{\epsilon}\left(1+\frac{\epsilon}{\kappa^2}\right)\Vert{\Vector{g}}\Vert_{0,\partial{G}}^2 \\
         & \lesssim \frac{\kappa^2}{\epsilon^2}\Vert{\Vector{J}}\Vert_{0,{G}}^2 + \frac{\kappa}{\epsilon}\Vert{\Vector{g}}\Vert_{0,\partial{G}}^2.
    \end{align*}
\end{proof}
The stability constant in Theorem \ref{thm:maxwellGerStable} will go to infinity as $\epsilon\rightarrow 0$. Following the approach in \cite[Theorem 3.1]{hiptmair2011stability}, we could obtain another stability result when $\epsilon/\kappa^2$ is small.
\begin{theorem}\label{thm:maxwellLitepStable}
    Let ${G}\subset \mathbb{R}^3$ be a bounded $C^2$-domain with characteristic length $L$, which is star-shaped with respect to a ball $B_\gamma(\Vector{x}_0)$ of radius $\gamma\sim L$, and assume that
    $\nabla\cdot\Vector{J}=0$. Then there exists a constant $C$ such that when $\frac{\epsilon L}{\kappa}\leq C$, the solution to \eqref{eq:shiftMaxwellVF}, denoted by $\Vector{E}$, satisfies
    \begin{equation}\label{eq:LitepStableDomain}
        \Vert{\nabla\times\Vector{E}}\Vert_{0,{G}}^2+\kappa^2\Vert{\Vector{E}}\Vert_{0,{G}}^2\lesssim L^2\Vert{\Vector{J}}\Vert_{0,{G}}^2 + L\Vert{\Vector{g}}\Vert_{0,\partial{G}}^2
    \end{equation}
    and
    \begin{equation}\label{eq:LitepStableBoundary}
        \kappa^2\Vert{\Vector{E}_T}\Vert_{0,\partial{G}}^2\lesssim L\Vert{\Vector{J}}\Vert_{0,{G}}^2 + \Vert{\Vector{g}}\Vert_{0,\partial{G}}^2.
    \end{equation}
\end{theorem}
\begin{proof}
    We will only focus on the case when $\Vector{g}\in H_T^{1/2}(\partial{G})$ and $\Vector{E}\in H^1(\mathbf{curl};{G})$ (see \cite[Lemma 3.2]{hiptmair2011stability}). Without loss of generality, we could set $\Vector{x}_0=0$. Ignoring the term $\Vert{\Vector{E}}\Vert_{0,{G}}$ in the imaginary part of \eqref{eq:proofGerStable1} and using Young's inequality give
    \begin{equation*}
        \begin{aligned}
            \kappa\Vert{\Vector{E}_T}\Vert_{0,\partial{G}}^2
             & \leq \Vert{\Vector{J}}\Vert_{0,{G}}\Vert{\Vector{E}}\Vert_{0,{G}} + \Vert{\Vector{g}}\Vert_{0,\partial{G}}\Vert{\Vector{E}}\Vert_{0,\partial{G}} \\
             & \leq \frac{1}{2\eta_1}\Vert{\Vector{J}}\Vert_{0,{G}}^2+\frac{\eta_1}{2}\Vert{\Vector{E}}\Vert_{0,{G}}^2+\frac{1}{2\kappa}\Vert{\Vector{g}}\Vert_{0,\partial{G}}^2 +\frac{\kappa}{2}\Vert{\Vector{E}_T}\Vert_{0,\partial{G}}^2,
        \end{aligned}
    \end{equation*}
    namely
    \begin{equation}\label{eq:proofLitepStableET}
        \kappa^2\Vert{\Vector{E}_T}\Vert_{0,\partial{G}}^2 \leq \Vert{\Vector{g}}\Vert_{0,\partial{G}}^2 + \frac{\kappa}{\eta_1}\Vert{\Vector{J}}\Vert_{0,{G}}^2 + \kappa\eta_1\Vert{\Vector{E}}\Vert_{0,{G}}^2.
    \end{equation}
    Then, taking $\xi=(\nabla\times {\Vector{E}})\times\Vector{x}$ in \eqref{eq:shiftMaxwellVF} and following the same procedure in the proof of Theorem 3.1 of \cite{hiptmair2011stability} will finally lead to
    \begin{align*}
         & \quad\frac{1}{2}\Vert{\nabla\times\Vector{E}}\Vert_{0,{G}}^2 + \frac{\kappa^2}{2}\Vert{\Vector{E}}\Vert_{0,{G}}^2 - \frac{1}{2}\int_{\partial{G}}\frac{\vert{\Vector{x}}\vert^2}{\Vector{x}\cdot\Vector{n}}\left(\kappa^2\vert{\Vector{E}_T}\vert^2+\vert{(\nabla\times\Vector{E})_T}\vert^2\right)\Diff S \\
         & \leq \mathop{\rm Re}\left(\int_{{G}}\Vector{J}\cdot((\nabla\times\overline{\Vector{E}})\times\Vector{x})\Diff V\right) - \epsilon\mathop{\rm Im}\left(\int_{{G}}\Vector{E}\cdot((\nabla\times\overline{\Vector{E}})\times\Vector{x})\Diff V\right).
    \end{align*}
    Therefore, using the fact that $(\nabla\times\Vector{E})\times\Vector{n} = \Vector{g}+\mathrm{i}\kappa\Vector{E}_T$, the inequality \eqref{eq:proofLitepStableET} and Young's inequality, there holds
    \begin{align}\label{eq:proofLitepStableCurl}
         & \quad\Vert{\nabla\times\Vector{E}}\Vert_{0,{G}}^2 + \kappa^2\Vert{\Vector{E}}\Vert_{0,{G}}^2 \cr
         & \leq 2L\Vert{\Vector{J}}\Vert_{0,{G}}\Vert{\nabla\times\Vector{E}}\Vert_{0,{G}} +  2\epsilon L\Vert{\Vector{E}}\Vert_{0,{G}}\Vert{\nabla\times\Vector{E}}\Vert_{0,{G}} \cr
         & +\frac{L^2}{\lambda}\left(\kappa^2\Vert{\Vector{E}_T}\Vert_{0,\partial{G}}^2+\Vert{(\nabla\times\Vector{E})_T}\Vert_{0,\partial{G}}^2\right) \cr
         & \leq \frac{L^2}{\eta_2} \Vert{\Vector{J}}\Vert_{0,{G}}^2 + \eta_2\Vert{\nabla\times\Vector{E}}\Vert_{0,{G}}^2 + \frac{\epsilon^2L^2}{\eta_3}\Vert{\Vector{E}}\Vert_{0,{G}}^2 + \eta_3\Vert{\nabla\times\Vector{E}}\Vert_{0,{G}}^2 \cr
         & +\frac{3\kappa^2 L^2}{\lambda}\Vert{\Vector{E}_T}\Vert_{0,\partial{G}}^2 + \frac{2 L^2}{\lambda} \Vert{\Vector{g}}\Vert_{0,\partial{G}}^2 \cr
         & \leq \left(\frac{L^2}{\eta_2}+\frac{3\kappa L^2}{\lambda\eta_1}\right)\Vert{\Vector{J}}\Vert_{0,{G}}^2 + \frac{5 L^2}{\lambda} \Vert{\Vector{g}}\Vert_{0,\partial{G}}^2 + (\eta_2+\eta_3)\Vert{\nabla\times\Vector{E}}\Vert_{0,{G}}^2 \cr
         & + \left(\frac{\epsilon^2 L^2}{\kappa^2\eta_3}+\frac{3 L^2\eta_1}{\kappa\lambda }\right)\kappa^2\Vert{\Vector{E}}\Vert_{0,{G}}^2.
    \end{align}
    Taking $\eta_1 = \frac{\kappa\lambda}{6L^2}$, $\eta_2 = \frac{1}{2}$ and $\eta_3 = \frac{\epsilon L}{\kappa}$ in \eqref{eq:proofLitepStableCurl}, we derive
    \begin{equation*}
        \left(\frac{1}{2}-\frac{\epsilon L}{\kappa}\right)\left(\Vert{\nabla\times\Vector{E}}\Vert_{0,{G}}^2+\kappa^2\Vert{\Vector{E}}\Vert_{0,{G}}^2\right)\leq \left(2L^2+\frac{18L^4}{\lambda^2} \right)\Vert{\Vector{J}}\Vert_{0,{G}}^2 + \frac{5L^2}{\lambda}\Vert{\Vector{g}}\Vert_{0,\partial{G}}^2.
    \end{equation*}
    Notice that $\gamma\sim L$, then when $\frac{\epsilon L}{\kappa}\leq\frac{1}{2}$ there holds
    \begin{equation*}
        \Vert{\nabla\times\Vector{E}}\Vert_{0,{G}}^2+\kappa^2\Vert{\Vector{E}}\Vert_{0,{G}}^2\lesssim L^2\Vert{\Vector{J}}\Vert_{0,{G}}^2 + L\Vert{\Vector{g}}\Vert_{0,\partial{G}}^2.
    \end{equation*}
    Moreover, taking $\eta_1=\frac{\kappa}{L}$ in \eqref{eq:proofLitepStableET} gives
    \begin{equation*}
        \kappa^2\Vert{\Vector{E}_T}\Vert_{0,\partial{G}}^2\lesssim L\Vert{\Vector{J}}\Vert_{0,{G}}^2 + \Vert{\Vector{g}}\Vert_{0,\partial{G}}^2.
    \end{equation*}
\end{proof}

In Theorem \ref{thm:maxwellLitepStable} we have assumed that ${G}$ is $C^2$-smooth and $\nabla\cdot\Vector{J} = 0$. In fact, these two requirements can be relaxed.
\begin{corollary}\label{coro:maxwellLitepStablePoly}
    Let ${G}\subset\mathbb{R}^3$ be a bounded polyhedron with characteristic lentgh $L$, which is star-shaped with respect to the ball $B_\gamma(\Vector{x}_0)$ of radius $\gamma\sim L$.
    Then for general $\Vector{J}\in L^2({G})^3$ and $\Vector{g}\in L^2_T(\partial{G})$, there exists a constant $C$ such that when $\epsilon L/\kappa\leq C$, the solution to \eqref{eq:shiftMaxwellVF}, denoted by $\Vector{E}$, satisfies
    \begin{equation}\label{eq:LitepStablePolyDomain}
        \Vert{\nabla\times\Vector{E}}\Vert_{0,{G}}^2+\kappa^2\Vert{\Vector{E}}\Vert_{0,{G}}^2\lesssim (\kappa^{-2}+L^2)\Vert{\Vector{J}}\Vert_{0,{G}}^2 + L\Vert{\Vector{g}}\Vert_{0,\partial{G}}^2
    \end{equation}
    and
    \begin{equation}\label{eq:LitepStablePolyBoundary}
        \kappa^2\Vert{\Vector{E}_T}\Vert_{0,\partial{G}}^2\lesssim L\Vert{\Vector{J}}\Vert_{0,{G}}^2 + \Vert{\Vector{g}}\Vert_{0,\partial{G}}^2.
    \end{equation}
\end{corollary}
\begin{proof}
    At first we assume that ${G}$ is $C^2$-smooth but $\nabla\cdot\Vector{J}\neq 0$. Define $\phi\in H_0^1({G})$ as the solution to the following equation
    \begin{equation*}
        -\Delta \phi = \nabla\cdot\Vector{J}.
    \end{equation*}
    It is obvious that $\nabla\cdot(\nabla\phi) = \Delta\phi = -\nabla\cdot\Vector{J}\in H^{-1}({G})$ and
    \begin{align*}
        \Vert{\nabla\cdot\Vector{J}}\Vert_{-1,{G}}
        = \sup_{w\in H_1^0({G})} \frac{|(\nabla\cdot\Vector{J},w)|}{\Vert{w}\Vert_{1,{G}}}
        = \sup_{w\in H_1^0({G})} \frac{|(\Vector{J},\nabla w)|}{\Vert{w}\Vert_{1,{G}}} \leq \Vert{\Vector{J}}\Vert_{0,{G}}.
    \end{align*}
    From Lax-Milgram theorem, we have
    \begin{equation*}
        \vert{\phi}\vert_{1,{G}}\lesssim \Vert{\nabla\cdot\Vector{J}}\Vert_{-1,{G}} \leq \Vert{\Vector{J}}\Vert_{0,{G}}.
    \end{equation*}
    Define $\widetilde{\Vector{E}}\coloneqq\Vector{E}-(\kappa^2+\mathrm{i}\epsilon)^{-1}\nabla\phi$, then $\widetilde{\Vector{E}}$ satisfies the following equation:
    \begin{equation*}
        \left\lbrace
        \begin{aligned}
             & \nabla\times(\nabla\times\widetilde{\Vector{E}}) -(\kappa^2+\mathrm{i}\epsilon)\widetilde{\Vector{E}}=\Vector{J} + \nabla\phi & & \text{in }{G}, \\
             & (\nabla\times\widetilde{\Vector{E}})\times\Vector{n}-\mathrm{i}\kappa\widetilde{\Vector{E}}_T=\Vector{g} & & \text{on }\partial{G},
        \end{aligned}
        \right.
    \end{equation*}
    where $\nabla\cdot(\Vector{J}+\nabla\phi) = 0$. Therefore \eqref{eq:LitepStableDomain} and \eqref{eq:LitepStableBoundary} hold for $\widetilde{\Vector{E}}$ from Theorem \ref{thm:maxwellLitepStable}. Then we can obtain
    \begin{equation*}
        \Vert{\nabla\times\Vector{E}}\Vert_{0,{G}}^2+\kappa^2\Vert{\Vector{E}}\Vert_{0,{G}}^2\lesssim (\kappa^{-2}+L^2)\Vert{\Vector{J}}\Vert_{0,{G}}^2 + L\Vert{\Vector{g}}\Vert_{0,\partial{G}}^2
    \end{equation*}
    and
    \begin{equation*}
        \kappa^2\Vert{\Vector{E}_T}\Vert_{0,\partial{G}}^2\lesssim L\Vert{\Vector{J}}\Vert_{0,{G}}^2 + \Vert{\Vector{g}}\Vert_{0,\partial{G}}^2.
    \end{equation*}

    If ${G}$ is a general polyhedron, the proof follows from \cite[Theorem 3.2]{hiptmair2011stability}, together with the above arguments.
\end{proof}

Combining Theorem \ref{thm:maxwellGerStable} and Corollary \ref{coro:maxwellLitepStablePoly}, we could obtain the stability result for any $\epsilon>0$.
\begin{theorem}\label{thm:maxwellStable}
    Let the assumptions made in Theorem \ref{thm:maxwellGerStable} and Corollary \ref{coro:maxwellLitepStablePoly} be satisfied, then the solution to \eqref{eq:shiftMaxwellVF}, denoted by $\Vector{E}$, satisfies
    \begin{equation}\label{eq:maxwellStable}
        \Vert{\Vector{E}}\Vert_{\imp,\kappa,{G}}^2\lesssim(1+(\kappa L)^{-1}+(\kappa L)^{-2})L^2\Vert{\Vector{J}}\Vert_{0,{G}}^2 + (1+(\kappa L)^{-1})L\Vert{\Vector{g}}\Vert_{0,\partial{G}}^2.
    \end{equation}
\end{theorem}
\begin{proof}
    Let $C$ be the constant defined in Corollary \ref{coro:maxwellLitepStablePoly}. When $\epsilon L/\kappa\leq C$, dividing both sides of \eqref{eq:LitepStablePolyBoundary} by $\kappa$
    and adding the result to \eqref{eq:LitepStablePolyDomain}, we can get \eqref{eq:maxwellStable}. On the other hand, when $\epsilon L/\kappa> C$, inserting it into \eqref{eq:maxwellGerStable}, we have
    \begin{equation*}
        \Vert{\nabla\times\Vector{E}}\Vert_{0,{G}}^2 + \kappa^2\Vert{\Vector{E}}\Vert_{0,{G}}^2 + \kappa\Vert{\Vector{E}_T}\Vert_{0,\partial{G}}^2 \lesssim L^2 \Vert{\Vector{J}}\Vert_{0,{G}}^2 + L \Vert{\Vector{g}}\Vert_{0,\partial{G}}^2.
    \end{equation*}
    This leads to \eqref{eq:maxwellStable}.
\end{proof}

Based on Theorem \ref{thm:maxwellStable}, we will use the so-called ``Schatz-type'' argument to establish a stability result of the discrete problem \eqref{eq:discreteMVF} in the rest of this section.
Such argument has also been used to build discrete stability results for Helmholtz equations in \cite{melenk2010convergence,graham2020domain}.

Define $S^*:H_\imp(\Vector{curl};{G}) \rightarrow  H_{\imp}(\Vector{curl};{G})$ such that, for any $\Vector{v}\in H_\imp(\Vector{curl};{G})$, the function $S^*\Vector{v}$ is the solution to the
following adjoint problem:
\begin{equation}\label{eq:S*def}
    a_{\epsilon}(\Vector{\xi},\mathcal{S}^*\Vector{v})=(\kappa^2+\mathrm{i}\epsilon)({\Vector{\xi}},{\Vector{v}})_{0,{G}}+\mathrm{i}\kappa\langle{\Vector{\xi}_T},{\Vector{v}_T}\rangle_{0,\partial{G}},\quad\forall\Vector{\xi}\in H_\imp(\Vector{curl};{G}).
\end{equation}
With the help of the operator $S^*$, we can define a quantity $\eta(V_h({G}))$ by
\begin{equation*}
    \eta(V_h(G))\coloneqq \sup_{0\neq\Vector{v}\in H^1({G})^3}\inf_{\Vector{w}\in V_h(G)}\frac{\Vert{S^*\Vector{v}-\Vector{w}}\Vert_{\imp,\kappa,{G}}}{\Vert{\Vector{v}}\Vert_{1,{G}}}.
\end{equation*}
\begin{theorem}\label{thm:discreteInfSupM}
    Let the assumption made in Theorem \ref{thm:maxwellStable} be satisfied.
    Then there exists a constant $C$ independent of $\kappa,\epsilon,L$ such that when $\kappa h \leq C$ and $\eta(V_h(G))\leq C$, the sesquilinear form $a_\epsilon(\cdot,\cdot)$ admits the following inf-sup condition.
    \begin{align*}
        \quad\inf_{0\neq\Vector{v}_h\in V_h({G})} \sup_{0\neq\Vector{w}_h\in V_h({G})}\frac{\vert{a_{\epsilon}(\Vector{v}_h,\Vector{w}_h)}\vert}{\Vert{\Vector{v}_h}\Vert_{\imp,\kappa,{G}}\Vert{\Vector{w}_h}\Vert_{\imp,\kappa,{G}}}
        \gtrsim\frac{1}{1+\kappa L}.
    \end{align*}
\end{theorem}
\begin{proof}
    For any $\Vector{v}_h\in V_h({G})\subset H_{\imp}(\Vector{curl};{G})$, there holds the following Helmholtz decomposition:
    \begin{equation*}
        \Vector{v}_h= \nabla q + \Vector{\phi},
    \end{equation*}
    where $q\in H^1({G})$ and $\Vector{\phi}\in H^1({G})^3$ satisfy
    \begin{equation}\label{eq:helmholtzDecomp}
        \Vert{q}\Vert_{1,{G}} + \Vert{\Vector{\phi}}\Vert_{1,{G}} \lesssim \Vert{\Vector{v}_h}\Vert_{\imp,\kappa,{G}}.
    \end{equation}
    Notice that $\Vector{v}_h$ is a N\'ed\'elec finite element function, then we have
    \begin{align*}
        \Vector{v}_h=\Vector{r}_h\Vector{v}_h = \Vector{r}_h(\nabla q) + \Vector{r}_h \Vector{\phi}=\nabla q_h+\Vector{\phi} - (\Vector{\phi}-\Vector{r}_h\Vector{\phi}),
    \end{align*}
    where $q_h\in Z_h({G})$. Since $\Vector{\phi}\in H^1({G})^3$ and $\nabla\times\Vector{\phi} = \nabla\times \Vector{v}_h$ is a piecewise polynomial, there holds the following estimate (refer to \cite[Lemma 3.1]{toselli2000overlapping})
    \begin{equation}\label{eq:proofMaxwellInfSupInterp}
        \Vert{\Vector{\phi}-\Vector{r}_h\Vector{\phi}}\Vert_{0,{G}}+h^{\frac{1}{2}}\Vert{(\Vector{\phi}-\Vector{r}_h\Vector{\phi})_T}\Vert_{0,\partial{G}}\lesssim h \Vert{\Vector{\phi}}\Vert_{1,{G}}.
    \end{equation}
    It is easy to see that
    \begin{equation*}
        a_{\epsilon}(\Vector{\xi},\nabla q_h) = -(\kappa^2+\mathrm{i}\epsilon)({\Vector{\xi}},{\nabla q_h})_{{G}}-\mathrm{i}\kappa\langle{\Vector{\xi}_T},{(\nabla q_h)_T}\rangle_{\partial{G}},\quad \forall \Vector{\xi}\in H_{\imp}({G}).
    \end{equation*}
    Then $S^*(\nabla q_h) = -\nabla q_h$. Therefore, defining $\Vector{w} = 2S^*(\nabla q_h+\Vector{\phi}) = -2\nabla q_h + 2S^*(\Vector{\phi})$, it's easy to check that $\Vector{w} = 2S^*(\Vector{v}_h+(\Vector{\phi}-\Vector{r}_h\Vector{\phi}))$. Then we have
    \begin{align*}
        a_{\epsilon}(\Vector{v}_h,\Vector{v}_h+\Vector{w})
         & = a_{\epsilon}(\Vector{v}_h,\Vector{v}_h) + 2a_{\epsilon}(\Vector{v}_h,S^*\Vector{v}_h) + 2a_{\epsilon}(\Vector{v}_h,S^*(\Vector{\phi}-\Vector{r}_h\Vector{\phi}))  \\
         & =\Vert{\nabla\times\Vector{v}_h}\Vert_{0,{G}}^2 + (\kappa^2+\mathrm{i}\epsilon)\Vert{\Vector{v}_h}\Vert_{0,{G}}^2 + \mathrm{i}\kappa\Vert{(\Vector{v}_h)_T}\Vert_{0,\partial{G}}^2  \\
         & + 2(\kappa^2+\mathrm{i}\epsilon)({\Vector{v}_h},{\Vector{\phi}-\Vector{r}_h\Vector{\phi}})_{{G}} + 2\mathrm{i}\kappa\langle{\Vector{v}_{h,T}},{(\Vector{\phi}-\Vector{r}_h\Vector{\phi})_T}\rangle_{\partial{G}}.
    \end{align*}
    Combining it with \eqref{eq:proofMaxwellInfSupInterp} gives
    \begin{equation*}
        \begin{aligned}
            \vert{a_\epsilon(\Vector{v}_h,\Vector{v}_h+\Vector{w})}\vert
             & \geq \frac{1}{\sqrt{2}}\Vert{\Vector{v}_h}\Vert_{\imp,\kappa,{G}}^2 -2\sqrt{2}\kappa^2\Vert{\Vector{v}_h}\Vert_{0,{G}}\Vert{\Vector{\phi}-\Vector{r}_h\Vector{\phi}}\Vert_{0,{G}} \\
             & - 2\kappa\Vert{\Vector{v}_{h,T}}\Vert_{0,\partial{G}}\Vert{(\Vector{\phi}-\Vector{r}_h\Vector{\phi})_T}\Vert_{0,\partial{G}} \\
             & = \frac{1}{\sqrt{2}}\Vert{\Vector{v}_h}\Vert_{\imp,\kappa,{G}}^2 -O(\kappa h)\kappa\Vert{\Vector{v}_h}\Vert_{0,{G}}\Vert{\Vector{\phi}}\Vert_{1,{G}} \\
             & -O((\kappa h)^\frac{1}{2})\kappa^{\frac{1}{2}}\Vert{(\Vector{v}_h)_T}\Vert_{0,\partial{G}}\Vert{\Vector{\phi}}\Vert_{1,{G}} \\
             & \gtrsim \left[1-O(\kappa h + (\kappa h)^\frac{1}{2})\right]\Vert{\Vector{v}_h}\Vert_{\imp,\kappa,{G}}^2,
        \end{aligned}
    \end{equation*}
    where the last inequality is derived from \eqref{eq:helmholtzDecomp}. Therefore we have
    \begin{equation*}
        \vert{a_\epsilon(\Vector{v}_h,\Vector{v}_h+\Vector{w})}\vert \gtrsim \Vert{\Vector{v}_h}\Vert_{\imp,\kappa,{G}}^2
    \end{equation*}
    provided that $\kappa h$ is small enough.

    Define $\Vector{w}_h = -2\nabla p_h + 2\Vector{\phi}_h\in V_h({G})$, where $\Vector{\phi}_h\in V_h({G})$ satisfies
    \begin{equation*}
        \Vector{\phi}_h = \mathop{\mathrm{arginf}}_{\Vector{\psi}_h\in V_h({G})}\Vert{S^*\Vector{\phi}-\Vector{\psi}_h}\Vert_{\imp,\kappa,{G}}.
    \end{equation*}
    Then $\Vector{w}_h$ is a good approximation of $\Vector{w}$, namely,
    \begin{equation*}
        \Vert{\Vector{w}-\Vector{w}_h}\Vert_{\imp,\kappa,{G}} = 2\Vert{\Vector{\phi}_h-\Vector{\phi}}\Vert_{\imp,\kappa,{G}}\leq 2\eta(V_h({G}))\Vert{\Vector{\phi}}\Vert_{1,{G}}\lesssim\eta(V_h({G}))\Vert{\Vector{v}_h}\Vert_{\imp,\kappa,{G}}.
    \end{equation*}
    Thus we can immediately derive
    \begin{align}
        \big\vert a_{\epsilon}(\Vector{v}_h,\Vector{v}_h+\Vector{w}_h)\big\vert
         & = \vert{a_{\epsilon}(\Vector{v}_h,\Vector{v}_h+\Vector{w}) + a_{\epsilon}(\Vector{v}_h,\Vector{w}_h-\Vector{w})}\vert \cr
         & \geq \vert{a_{\epsilon}(\Vector{v}_h,\Vector{v}_h+\Vector{w})}\vert-O(1)\eta(V_h({G}))\Vert{\Vector{v}_h}\Vert_{\imp,\kappa,{G}}^2 \cr
         & \gtrsim \left[1-O(1)\eta(V_h({G}))\right]\Vert{\Vector{v}_h}\Vert_{\imp,\kappa,{G}}^2 \cr
         &\gtrsim \Vert{\Vector{v}_h}\Vert_{\imp,\kappa,{G}}^2\label{eq:proofDiscreteInfSupM1}
    \end{align}
    provided that $\eta(V_h(G))$ is small enough. On the other hand, from Theorem \ref{thm:maxwellStable}, \eqref{eq:proofMaxwellInfSupInterp}, \eqref{eq:helmholtzDecomp} and $\kappa h\lesssim 1$,
    there holds
    \begin{align*}
        \Vert{\Vector{w}}\Vert_{\imp,\kappa,{G}}^2
         & \lesssim \left(1+\kappa L+(\kappa L)^2\right)\kappa^2\Vert{\Vector{v}_h+ (\Vector{\phi}-\Vector{r}_h\Vector{\phi})}\Vert_{0,{G}}^2 \\
         & +\left(1+\kappa L\right)\kappa\Vert{(\Vector{v}_h+(\Vector{\phi}-\Vector{r}_h\Vector{\phi}))_T}\Vert_{0,\partial{G}}^2 \\
         & \lesssim \left(1+\kappa L+(\kappa L)^2\right)\kappa^2\left(\Vert{\Vector{v}_h}\Vert_{0,{G}}^2+  h^2\Vert{\Vector{v}_h}\Vert_{\imp,\kappa,{G}}^2\right) \\
         & + \left(1+\kappa L\right)\kappa\left(\Vert{\Vector{v}_{h,T}}\Vert_{0,\partial{G}}^2+h\Vert{\Vector{v}_h}\Vert_{\imp,\kappa,{G}}^2\right) \\
         & \lesssim \left(1+\kappa L+(\kappa L)^2\right) \Vert{\Vector{v}_h}\Vert_{\imp,\kappa,{G}}^2 \\
         & \lesssim \left(1+\kappa L\right)^2\Vert{\Vector{v}_h}\Vert_{\imp,\kappa,{G}}^2,
    \end{align*}
    namely
    \begin{equation*}
        \Vert{\Vector{w}}\Vert_{\imp,\kappa,{G}} \lesssim \left(1+\kappa L\right)\Vert{\Vector{v}_h}\Vert_{\imp,\kappa,{G}}.
    \end{equation*}
    Then we have
    \begin{align}\label{eq:proofDiscreteInfSupM2}
        \Vert{\Vector{v}_h+\Vector{w}_h}\Vert_{\imp,\kappa,{G}}
         & \leq \Vert{\Vector{v}_h}\Vert_{\imp,\kappa,{G}}+\Vert{\Vector{w}}\Vert_{\imp,\kappa,{G}} + \Vert{\Vector{w}-\Vector{w}_h}\Vert_{\imp,\kappa,{G}} \cr
         & \lesssim \left(1+\kappa L +\eta(V_h({G}))\right)\Vert{\Vector{v}_h}\Vert_{\imp,\kappa,{G}} \cr
         & \lesssim \left(1+\kappa L\right)\Vert{\Vector{v}_h}\Vert_{\imp,\kappa,{G}}.
    \end{align}
    Combining \eqref{eq:proofDiscreteInfSupM1} and \eqref{eq:proofDiscreteInfSupM2} finally leads to
    \begin{align*}
         & \quad\inf_{0\neq\Vector{v}_h\in V_h({G})} \sup_{0\neq\Vector{\psi}_h\in V_h({G})}\frac{\vert{a_{\epsilon}(\Vector{v}_h,\Vector{\psi}_h)}\vert}{\Vert{\Vector{v}_h}\Vert_{\imp,\kappa,{G}}\Vert{\Vector{\psi}_h}\Vert_{\imp,\kappa,{G}}} \\
         & \geq \inf_{0\neq\Vector{v}_h\in V_h({G})} \frac{\vert{a_{\epsilon}(\Vector{v}_h,\Vector{v}_h+\Vector{w}_h)}\vert}{\Vert{\Vector{v}_h}\Vert_{\imp,\kappa,{G}}\Vert{\Vector{v}_h+\Vector{w}_h}\Vert_{\imp,\kappa,{G}}}
        \gtrsim\frac{1}{1+\kappa L}.
    \end{align*}
\end{proof}

The value of $\eta(V_h({G}))$ reflects the extent to which $V_h({G})$ approximates the solution to the adjoint problem \eqref{eq:S*def}. Theorem \ref{thm:discreteInfSupM} holds only when $\eta(V_h({G}))$ is small enough. We will make the following assumption.
\begin{assumption}\label{assump:eta}
    For any $C>0$, there exists $\bar{h}(\kappa,p)$ such that if the mesh size $h$ is chosen such that $h\leq\bar{h}(\kappa,p)$, then there holds $\eta(V_h(G))\leq C$.
\end{assumption}
\begin{remark}
    According to \cite[Theorem 7.3]{melenk2023wavenumber}, for pure Maxwell problems posed on analytic domain, Assumption \ref{assump:eta} holds with $\bar{h}(\kappa,p)\sim \kappa^{-\frac{p+1}{p}}$, which can be easily generalized to problems with absorption. Unfortunately, to the best of our knowledge, there are still no results on the wavenumber-explicit representation of $\bar{h}(\kappa,p)$ for problem posed on a general domain. Nevertheless, from \cite{monk2003finite,gatica2012finite}, we know that such $\bar{h}(\kappa,p)$ exists.
\end{remark}

\begin{corollary}\label{coro:discreteStable}
    Let the assumptions made in Theorem \ref{thm:maxwellStable} and Assumption \ref{assump:eta} be satisfied. We have
    \begin{equation}\label{eq:discreteStableM1}
        \Vert{\Vector{E}_h}\Vert_{\imp,\kappa,{G}}\lesssim\min\bigg\{1+\kappa L,\frac{\kappa^2}{\epsilon}\bigg\}\sup_{\Vector{\xi}_h\in V_h({G})}\frac{\vert{\Vector{F}(\Vector{\xi})}\vert}{\Vert{\Vector{\xi}}\Vert_{\imp,\kappa,{G}}},
    \end{equation}
    where $\Vector{E}_h$ is the solution to \eqref{eq:discreteMVF}.
    Moreover, if ${G}$ is not star-shaped, we can only obtain
    \begin{equation}\label{eq:discreteStableM2}
        \Vert{\Vector{E}_h}\Vert_{\imp,\kappa,{G}}\lesssim\frac{\kappa^2}{\epsilon}\sup_{\Vector{\xi}_h\in V_h({G})}\frac{\vert{\Vector{F}(\Vector{\xi})}\vert}{\Vert{\Vector{\xi}}\Vert_{\imp,\kappa,{G}}}.
    \end{equation}
\end{corollary}
\begin{proof}
    The estimate \eqref{eq:discreteStableM1} follows from Lemma \ref{lemma:contCoerM} and Theorem \ref{thm:discreteInfSupM}. When ${G}$ is not star-shaped, the results in Theorem \ref{thm:discreteInfSupM} fail and one can only obtain \eqref{eq:discreteStableM2} from Lemma \ref{lemma:contCoerM}.
\end{proof}

\begin{remark}\label{remark:adjointStableM}
    The results in Theorem \ref{thm:maxwellStable} and Corollary \ref{coro:discreteStable} also hold for adjoint problem.
\end{remark}

\section{A hybrid Schwarz preconditioner}\label{sec:preconditioner}
In this section, we are devoted to constructing a two-level overlapping domain decomposition preconditioner $B_\epsilon^{-1}$ for $A_\epsilon$. This preconditioner is similar to that proposed in \cite{hu2024novel}.
\subsection{One level preconditioner}\label{subsec:1level}
Let $\overline{\Omega}=\cup_{l=1}^N \overline{\Omega_l}$ be an overlapping domain decomposition of $\Omega$. Each $\Omega_l$ consists of a union of elements of the mesh $\mathcal{T}_h$. In the following, we will make a few assumptions on the subdomains.

At first, we assume that the subdomains are shape-regular Lipschitz polyhedra of diameter $H_l=\mathop{\rm diam}(\Omega_l)$. We then denote the subdomain size by $H\coloneqq\max\{H_l:l = 1,...,N\}$ and the overlap by $\delta$. For a subdomain $\Omega_l$, define
\begin{equation*}
    \Lambda(l)\coloneqq \{j:~ \Omega_j\cap\Omega_l \neq \emptyset\}.
\end{equation*}
Namely, $\Lambda(l)$ consists of the indices of the subdomains that intersect with $\Omega_l$.
We will make the following assumptions.
\begin{assumption}[Finite Covering]\label{assump:finiteOverlap}
    The number of indices in $\Lambda(l)$ is uniformly bounded with respect to $l$, i.e.,
    \begin{equation*}
        \#\Lambda(l)\lesssim 1,\quad \forall 1\leq l\leq N.
    \end{equation*}
\end{assumption}
\begin{assumption}[Star-shaped Subdomains]
    Every subdomain $\Omega_l$ is star-shaped with respect to a ball with radius $\sim H$.
\end{assumption}

For each $\Omega_l$, we could define the natural restriction $R_l:V_h(\Omega)\rightarrow V_h(\Omega_l)$ by $R_l {\bf v}_h = {\bf v}_h|_{\Omega_l}$ for any ${\bf v}_h\in V_h(\Omega)$. Then the N\'ed\'elec finite element space on $\Omega_l$ can be defined by
$$V_h(\Omega_l)\coloneqq \{R_l{\bf v}_h:~{\bf v}_h\in V_h(\Omega)\}.$$
We can also define a prolongation operator $E_l:V_h(\Omega_l)\rightarrow V_h(\Omega)$.
Let $\xi$ denote an edge, or a face or an element in the mesh $\mathcal{T}_h$, and let $M_{\xi}^i({\bf v}_h), 1\leq i\leq n_\xi$ denote a N\'ed\'elec moment degree of freedom of ${\bf v}_h\in V_h(\Omega)$ corresponding to $\xi$, where $n_\xi$ is the number of the degrees of freedom associated with $\xi$ (see \cite{monk2003finite} for the precise definition). We use $\mathcal{N}_h(\Omega)$ (resp. $\mathcal{N}_h(\Omega_l)$) to denote the set of all edges/faces/elements lies in $\overline{\Omega}$ (rep. $\overline{\Omega}_l$).
For
${\bf v}_{h,l}\in V_h(\Omega_l)$ and $\xi\in \mathcal{N}_h(\Omega)$,
we could define $E_l\Vector{v}_{h,l}$ as

\begin{equation*}
    M_{\xi}^i(E_l\Vector{v}_{h,l}) =
    \left\lbrace
        \begin{aligned}
            &M_{\xi}^i(\Vector{v}_{h,l}), & & \text{if}~~\xi\in\mathcal{N}_h(\Omega_l)\\
            &~~~~0, & & \text{if}~~\xi\notin\mathcal{N}_h(\Omega_l)
        \end{aligned}
    \right.
    ,\quad 1\leq i\leq n_\xi
\end{equation*}

In the next, we will introduce a local impedance problem on each $\Omega_l$. Define the local sesquilinear form $a_{\epsilon,l}(\cdot,\cdot)$ by
\begin{align*}
    a_{\epsilon,l}(\Vector{v}_l,\Vector{w}_l)
    &\coloneqq ({\nabla\times \Vector{v}_l},{\nabla\times{\Vector{w}}_l})_{\Omega_l} - (\kappa^2+\mathrm{i}\epsilon)({\Vector{v}_l},{{\Vector{w}}_l})_{\Omega_l}\\
    &-\mathrm{i}\kappa\langle{\Vector{v}_{l,T}},{{\Vector{w}}_{l,T}}\rangle_{\partial\Omega_l}\quad\quad\forall \Vector{v}_l,\Vector{w}_l\in H_\imp(\Vector{curl};\Omega_l).
\end{align*}
Then the local discrete variational problem writes: find $\Vector{E}_{h,l}\in V_h(\Omega_l)$ such that
\begin{equation}\label{eq:localDiscreteMVF}
    a_{\epsilon,l}(\Vector{E}_{h,l},\Vector{\xi}_{h,l}) = \Vector{F}_l(\Vector{\xi}_{h,l}),\quad \forall \Vector{\xi}_{h,l}\in V_h(\Omega_l).
\end{equation}
And an operator $A_{\epsilon,l}:V_h(\Omega_l)\rightarrow V_h(\Omega_l)$ could be defined by
\begin{equation*}
    ({A_{\epsilon,l} \Vector{v}_{h,l}},{\Vector{w}_{h,l}})_{\Omega_l} = a_{\epsilon,l}(\Vector{v}_{h,l},\Vector{w}_{h,l}),\quad\forall \Vector{v}_{h,l},\Vector{w}_{h,l}\in V_h(\Omega_l).
\end{equation*}

As in \cite{graham2020domain,hu2024novel}, we introduce a set of functions $\{\chi_l\}_{l=1}^N$ forming a partition of unity to connect the local problems.
Assume that $\chi_l: \overline{\Omega}\rightarrow \mathbb{R}$ satisfies
\begin{equation*}
    \supp \chi_l \subseteq \overline{\Omega}_l,\quad 0\leq\chi_l\leq 1, \quad \sum_{l=1}^{N}\chi_l = 1
\end{equation*}
and
\begin{equation*}
    \vert\nabla\chi_l\vert\lesssim\delta^{-1}.
\end{equation*}
With the help of $\chi_l$, we could define a weighted operator $D_l:V_h(\Omega)\rightarrow V_h(\Omega_l)$ by
\begin{equation*}
    D_l(\Vector{v}_h) = \widetilde{\Vector{r}}_h(\chi_l \Vector{v}_h),\quad \forall \Vector{v}_h\in V_h(\Omega),
\end{equation*}
where $\widetilde{\Vector{r}}_h$ is a special ``quasi-N\'ed\'elec interpolation" operator introduced in Appendix \ref{sec:appendixA}.

Finally, the one-level weighted additive Schwarz preconditioner writes
\begin{equation}\label{eq:1levelpreM}
    B_{\epsilon,{\rm WASI}}^{-1} = \sum_{l=1}^{N} D_l E_l A_{\epsilon,l}^{-1}E_l^*.
\end{equation}

\begin{remark} In theory we can also use the standard ``N\'ed\'elec interpolation" operator ${\bf r}_h$ to define the weight operator $D_l$.
    There is no essential difference between the uses of $\widetilde{\Vector{r}}_h$ and $\Vector{r}_h$ in the convergence analysis in Section \ref{sec:convergence}. But the implementation of $\widetilde{\Vector{r}}_h$
    is more  economical than that of $\Vector{r}_h$ since there is no need to calculate numerical integrals for $\widetilde{\Vector{r}}_h$ (see Lemma \ref{lemma:newInterpApply}).
\end{remark}

\subsection{An adaptive coarse space}
In \cite{hu2024novel}, a novel adaptive coarse space is designed for the Helmholtz equations. In this section, we will follow the same idea to build a two-level adaptive Schwarz preconditioner
for time-harmonic Maxwell equations.

Set $\Gamma_l\coloneqq\partial\Omega_l\backslash\partial\Omega$, and let ${\bf n}_l$ denote the unit outward normal on $\Gamma_l$.
Define $V_h(\Gamma_l)$ as the tangential trace space consisting of the tangential traces of the finite element functions in $V_h(\Omega_l)$, i.e.,
\begin{equation*}
    V_h(\Gamma_l) \coloneqq \big\{(\Vector{v}_{h}\times{\bf n}_l)|_{\Gamma_l}:~\Vector{v}_{h}\in V_h(\Omega_l)\big\}.
\end{equation*}
Define ${\mathcal H}_{\epsilon, l}:V_h(\Gamma_l)\rightarrow V_h(\Omega_l)$ by
\begin{equation*}
    a_{\epsilon,l}({\mathcal H}_{\epsilon, l}(\Vector{\lambda}_h),\Vector{w}_h) =\langle\Vector{\lambda}_h,\Vector{w}_{h}\rangle_{\Gamma_l},\quad \forall \Vector{w}_h\in V_{h}(\Omega_l),
\end{equation*}
where ${ \Vector{\lambda}}_h\in V_h(\Gamma_l)$ is any impedance data on $\Gamma_l$. The well-posedness of ${\mathcal H}_{\epsilon, l}$ is guaranteed by Theorem \ref{thm:maxwellStable} and we could define the discrete Maxwell-harmonic space
\begin{equation}\label{eq:discreteMHarmonic}
    V_h^{\partial}(\Omega_l)\coloneqq \big\{ \Vector{v}_h\in V_h(\Omega_l):~\Vector{v}_h={\mathcal H}_{\epsilon, l}(\Vector{\lambda}_h),\quad\forall\Vector{\lambda}_h\in V_h(\Gamma_l)\big\}.
\end{equation}
This space is also well-defined. Further, we can define the weighted Maxwell-harmonic space
\begin{equation*}
    V^{\partial}_{h}(\Omega):=\bigg\{\Vector{v}_h=\sum_{l=1}^N D_l E_l \Vector{v}_l:~\Vector{v}_l\in V_{h}^{\partial}(\Omega_l)\bigg\}.
\end{equation*}

We then define the following generalized eigenvalue problem on $V_h^\partial(\Omega_l)$: find $\Vector{\xi}\in V_h^{\partial}(\Omega_l)$ such that
\begin{equation}\label{eq:geMEigenvalue}
    (D_l E_l \Vector{\xi},D_l E_l \Vector{\theta})_{\imp,\kappa,\Omega_l} = \lambda (\Vector{\xi},\Vector{\theta})_{\imp,\kappa,\Omega_l},\quad \forall\Vector{\theta}\in V_h^{\partial}(\Omega_l).
\end{equation}
It is clear that the bilinear forms appeared in \eqref{eq:geMEigenvalue} are Hermitian positive semi-definite. Then the eigenvalues $\{\lambda_l^i\}_{i=1}^{m_l}$ are real non-negative, where $m_l=\mathop{\rm dim}V_h^\partial(\Omega_l)$. We assume that the eigenvalues are arranged in ascending order,
$$\lambda_l^1\leq\lambda_l^2\leq...\leq\lambda_l^{m_l}$$
and that the eigenvectors $\{\Vector{\xi}_l^i\}_{i=1}^{m_l}$ satisfy
\begin{equation*}
    (\Vector{\xi}_l^i,\Vector{\xi}_l^j)_{\imp,\kappa,\Omega_l}=
    \left\lbrace
    \begin{aligned}
        &1,&& i=j, \\
        &0,&& i\neq j.
    \end{aligned}
    \right.
\end{equation*}
For a given tolerance $\rho\in(0,1)$, we can define the local coarse space $V_{h,0}^\rho(\Omega_l)$ as
\begin{equation*}
    V^{\rho}_{h,0}(\Omega_l) \coloneqq \mathop{\rm span}\{\Vector{\xi}_l^i:~ \theta_l^i\geq \rho^2\}.
\end{equation*}
Since every $\Vector{v}\in V^\partial_h(\Omega_l)$ admits the decomposition
\begin{equation}\label{eq:eigenDecomp}
    \Vector{v} = \sum_{i=1}^{m_l} (\Vector{v},\Vector{\xi}_l^i)_{\imp,\kappa,\Omega_l}\Vector{\xi}_l^i,
\end{equation}
we can define the projection $\Pi_l^\rho:V_h^\partial(\Omega_l)\rightarrow V_{h,0}^\rho(\Omega_l)$ by
\begin{equation*}
    \Pi_l^\rho \Vector{v} \coloneqq \sum_{\lambda_l^i\geq \rho^2}(\Vector{v},\Vector{\xi}_l^i)_{\imp,\kappa,\Omega_l}\Vector{\xi}_l^i,\quad \forall \Vector{v}\in V^\partial_h(\Omega_l).
\end{equation*}

The following approximation result can be verified in the standard manner.
\begin{lemma}\label{lemma:eigenEstimate}
    There holds the stability estimate
    \begin{equation}\label{eq:eigenMEstimate}
        \Vert D_l E_l(I-\Pi_l^\rho)\Vector{v}_h\Vert_{\imp,\kappa,\Omega_l}\leq\rho\Vert \Vector{v}_h\Vert_{\imp,\kappa,\Omega_l},\quad\forall \Vector{v}_h\in V^\partial_h(\Omega_l).
    \end{equation}
\end{lemma}

Under the above preparation, the coarse space is defined by
\begin{equation*}
    V^{\rho}_{h,0}(\Omega)\coloneqq \bigg\{\Vector{v}_h=\sum_{l=1}^N D_l E_l \Vector{v}_l:~\Vector{v}_l\in V_{h,0}^{\rho}(\Omega_l)\bigg\}.
\end{equation*}
Define $A^{\rho}_{\epsilon,0}:V^{\rho}_{h,0}(\Omega)\rightarrow V^{\rho}_{h,0}(\Omega)$ as the restriction of $A_\epsilon$ on $V^{\rho}_{h,0}(\Omega)$, i.e.
\begin{equation*}
    (A^{\rho}_{0} \Vector{v}_0,\Vector{w}_0) = a_{\epsilon}(\Vector{v}_0,\Vector{w}_0),\quad\forall \Vector{v}_0,\Vector{w}_0\in V^{\rho}_{h,0}(\Omega).
\end{equation*}
Let $E_0:V^{\rho}_{h,0}(\Omega)\rightarrow V_h(\Omega)$ be the identical lifting operator and $E_0^*$ be the adjoint operator of $E_0$ with respect to $L^2$ inner product.

Finally, the two-level weighted additive Schwarz preconditioner writes
\begin{equation}\label{eq:2levelPreM}
    B_{\epsilon}^{-1} = (I-E_0 (A^{\rho}_{0})^{-1} E_0^* A_{\epsilon})B_{\epsilon, {\rm WASI}}^{-1} + E_0 (A^{\rho}_{0})^{-1} E_0^*.
\end{equation}

\section{The main convergence results}\label{sec:convergence}
In this section, we will give a convergence analysis of the two-level preconditioner introduced in Section \ref{sec:preconditioner}. Since the preconditioner \eqref{eq:2levelPreM} is similar to that introduced in \cite{hu2024novel} in form, we follow the same technique developed in \cite{hu2024novel} to complete the convergence analysis by using the $inf$-$sup$ condition given in Theorem \ref{thm:discreteInfSupM}.
As we will see, there are some detailed differences  between the proofs of Lemma \ref{lemma:vhpBoundM} and the corresponding result in \cite{hu2024novel}. Here we have to estimate the $L^2_T$-norm of specific functions
with the help of discrete $L^2$-norm derived in Appendix \ref{sec:appendixB} since there is no trace inequality for $H(curl)$ functions in general.

In order to analyze the spectral properties of the preconditioned system $B_\epsilon^{-1} A_\epsilon$, we first define several projection operators. For $l=1,...,N$, the local projection operators $P_{\epsilon,l}: V_h(\Omega)\rightarrow V_h(\Omega_l)$ are defined as follows: for each $\Vector{v}_h\in V_h(\Omega)$, $P_{\epsilon,l} \Vector{v}_h\in V_h(\Omega_l)$ is defined as the solution to the local problem
\begin{equation}\label{eq:localProj}
    a_{\epsilon,l}(P_{\epsilon,l}\Vector{v}_h,\Vector{w}_{h,l}) = a_{\epsilon}(\Vector{v}_h,E_l \Vector{w}_{h,l}),\quad \forall \Vector{w}_{h,l}\in V_h(\Omega_l).
\end{equation}
It is easy to see that $P_{\epsilon,l}=A_{\epsilon,l}^{-1} E_l^* A_{\epsilon}$. The coarse projection operator $P^{\rho}_{\epsilon,0}: V_h(\Omega)\rightarrow V_h(\Omega)$ could be defined by the same way. For each $\Vector{v}_h\in V_h(\Omega)$, define $P^{\rho}_{\epsilon,0} \Vector{v}_h=E_0 \Vector{\zeta}_0$ with some $\Vector{\zeta}_0\in V_{h,0}^\rho$ satisfying
\begin{equation}\label{eq:coarseProj}0
    a_{\epsilon}(E_0\Vector{\zeta}_0,E_0 \Vector{w}_{h,0}) = a_{\epsilon}(\Vector{v}_h, E_0 \Vector{w}_{h,0}),\quad \forall \Vector{w}_{h,0}\in V^{\rho}_{h,0}(\Omega).
\end{equation}
Then we have $P^{\rho}_{\epsilon,0}=E_0 (A^{\rho}_{\epsilon,0})^{-1} E_0^* A_{\epsilon}$.

The well-posedness of \eqref{eq:localProj} and \eqref{eq:coarseProj} is guaranteed by \eqref{eq:coerciveM}. Finally, the global projection operator $P^{\rho}_{\epsilon}: V_h(\Omega)\rightarrow V_h(\Omega)$ is defined by
\begin{equation}\label{eq:globalProj}
    P^{\rho}_{\epsilon}= (I - {P^{\rho}_{\epsilon,0}})\sum_{l=1}^{N} D_l E_l P_{\epsilon,l} + {P^{\rho}_{\epsilon,0}}.
\end{equation}
It is easy to check that $P^{\rho}_{\epsilon} = B_\epsilon^{-1}A_\epsilon$.

\subsection{Some auxiliary results}
For $\Vector{v}_h\in V_h(\Omega)$, we could define its discrete $L^2$ norm in terms of the values of degrees of freedom, i.e.
\begin{equation*}
    \Vert{\Vector{v}_h}\Vert_{0,h,\Omega}^2 \coloneqq h^{d-2}\sum_{\xi\in \mathcal{N}_h(\Omega)}\sum_{i=1}^{n_\xi}\vert M_\xi^i(\Vector{v}_h)\vert^2.
\end{equation*}
It's easy to check that
\begin{equation}\label{eq:discreteL2Norm}
    \Vert{\Vector{v}_h}\Vert_{0,h,\Omega}\sim \Vert{\Vector{v}_h}\Vert_{0,\Omega}
\end{equation}
by scaling argument (see Appendix \ref{sec:appendixB} for the detailed proof).

As in \cite{hu2024novel}, for any $\Vector{v}_h\in V_h(\Omega)$, $P_\epsilon^\rho \Vector{v}_h$ can be reformulated as
\begin{equation}\label{eq:globalProjReformM}
    P^{\rho}_{\epsilon} \Vector{v}_h = (I-P^{\rho}_{0})\sum_{l=1}^{N}D_l E_l \Vector{v}_{h,l}^{\partial} + \Vector{v}_h,
\end{equation}
where
$\Vector{v}_{h,l}^\partial \coloneqq P_{\epsilon,l} \Vector{v}_h  -R_l \Vector{v}_h\in V^\partial_h(\Omega_l)$.

Thanks for Assumption \ref{assump:finiteOverlap}, there hold the following two lemmas (whose proofs are almost the same as that given in \cite{hu2024novel}).
\begin{lemma}\label{lemma:sumEnergyBoundM}
    Let $\Vector{v}_l\in H_{\imp}(\Vector{curl};\Omega)$ satisfy $\mathop{\rm supp}{\Vector{v}_l}\subseteq \overline{\Omega}_l$. Under Assumption  \ref{assump:finiteOverlap}, we can obtain
    \begin{equation}\label{eq:sumEnergyBoundM}
        \bigg\Vert{\sum_{l=1}^{N}\Vector{v}_l}\bigg\Vert_{\imp,\kappa,\Omega}^2\lesssim \sum_{l=1}^{N}\Vert{\Vector{v}_l}\Vert_{\imp,\kappa,\Omega}^2.
    \end{equation}
    Moreover, the right hand side of \eqref{eq:sumEnergyBoundM} satisfies
    \begin{equation*}
        \Vert{\Vector{v}_l}\Vert_{\imp,\kappa,\Omega} = \Vert{\Vector{v}_l}\Vert_{\imp^-,\kappa,\Omega_l}\leq\Vert{\Vector{v}_l}\Vert_{\imp,\kappa,\Omega_l}.
    \end{equation*}
\end{lemma}
\begin{lemma}\label{lemma:rlStableM}
    If Assumption \ref{assump:finiteOverlap} is satisfied, there holds
    \begin{equation*}
        \sum_{l=1}^N \Vert{\Vector{v}}\Vert_{\imp^-,\kappa,\Omega_l}^2\lesssim \Vert{\Vector{v}}\Vert_{\imp,\kappa,\Omega}^2,\quad \forall \Vector{v}\in H_\imp(\Vector{curl};\Omega).
    \end{equation*}
\end{lemma}

Define
\begin{equation}\label{eq:vhpM}
    \Vector{v}_h^\partial = \sum_{l=1}^{N}D_l E_l \Vector{v}_{h,l}^\partial.
\end{equation}
Since $\Vector{v}_{h,l}^\partial\in V_h^\partial(\Omega_l)$, we have $\Vector{v}_h^\partial\in V_h^\partial(\Omega)$. Then from Lemma \ref{lemma:eigenEstimate}, $\Vector{v}_h^\partial$ could be well approximated by functions in coarse space $V^{\rho}_{h,0}(\Omega)$.
\begin{lemma}\label{lemma:vhpApproxM}
    Let Assumption \ref{assump:finiteOverlap} be satisfied and $\Vector{z}_h = \sum_{l=1}^{N}D_l E_l \Pi_{l}^\rho \Vector{v}_{h,l}^\partial$. Then $\Vector{z}_h\in V^{\rho}_{h,0}(\Omega)$ and we can obtain
    \begin{equation}\label{eq:vhpApproxM}
        \Vert{\Vector{v}_h^\partial - \Vector{z}_h}\Vert_{\imp,\kappa,\Omega}\lesssim \rho \big(\sum_{l=1}^{N}\Vert{\Vector{v}_{h,l}^\partial}\Vert_{\imp,\kappa,\Omega_l}^2\big)^{\frac{1}{2}}.
    \end{equation}
\end{lemma}
Furthermore, to bound $\Vert{\Vector{v}_{h,l}^\partial}\Vert_{\imp,\kappa,\Omega_l}$, define $\widetilde{\Omega}_l$ as
\begin{equation}\label{eq:bdryStripe}
    \overline{\widetilde{\Omega}_l} = \mathop\bigcup_{\tau\in \widetilde{\mathcal T}_{h,l}} \overline{\tau},\quad \text{where }~~\widetilde{\mathcal T}_{h,l}\coloneqq\big\{\tau\in\mathcal{T}_h:~ \overline{\tau}\cap\overline{\Omega}_l\neq\emptyset\big\}.
\end{equation}
The last lemma in this subsection provides a bound of $\Vert{\Vector{v}_{h,l}^\partial}\Vert_{\imp,\kappa,\Omega_l}$ in terms of $\Vert{\Vector{v}_h}\Vert_{\imp^-,\kappa,\widetilde{\Omega}_l}$.
\begin{lemma}\label{lemma:vhpBoundM}
    Under Assumption \ref{assump:meshsize} and Assumption \ref{assump:finiteOverlap}, there holds
    \begin{equation}\label{eq:vhpBoundM}
        \Vert{\Vector{v}_{h,l}^\partial}\Vert_{\imp,\kappa,\Omega_l} \lesssim \min\bigg\{ {1+\kappa H_l,\frac{\kappa^2}{\epsilon}}\bigg\} (\kappa h)^{-\frac{1}{2}} \Vert{\Vector{v}_h}\Vert_{\imp^-,\kappa,\widetilde{\Omega}_l}.
    \end{equation}
\end{lemma}
\begin{proof}
    For any $\Vector{w}_{h,l}\in V_{h}(\Omega_l)$, it's easy to check that
    \begin{align}\label{eq:vhpImpDataM}
        a_{\epsilon,l}(\Vector{v}_{h,l}^\partial,\Vector{w}_{h,l})
         & = a_{\epsilon,l}(P_{\epsilon,l}\Vector{v}_h,\Vector{w}_{h,l}) - a_{\epsilon,l}(R_l \Vector{v}_h, \Vector{w}_{h,l}) \cr
         & = a_{\epsilon}(\Vector{v}_h, E_l \Vector{w}_{h,l}) - a_{\epsilon,l}(R_l \Vector{v}_h, \Vector{w}_{h,l})          \cr
         & \eqqcolon \Vector{F}_l(\Vector{w}_{h,l}),
    \end{align}
    which means that $\Vector{v}_{h,l}^\partial$ admits a local impedance problem with data $\Vector{F}_l$. Noting the discrete stability result, it suffices to estimate the bound of $|\Vector{F}_l(w_{h,l})|$.

    At first, from triangle inequality, there holds
    \begin{equation}\label{eq:reformFlM}
        \vert\Vector{F}_l(w_{h,l})\vert \leq \vert a_{\epsilon}(\Vector{v}_h, E_l \Vector{w}_{h,l})\vert + \vert a_{\epsilon,l}(R_l \Vector{v}_h, \Vector{w}_{h,l})\vert.
    \end{equation}
    For the estimate of the first term of the right hand side of \eqref{eq:reformFlM}, noting that $E_l \Vector{w}_{h,l}$ vanishes on $\Omega\backslash\widetilde{\Omega}_l$, we can derive
    \begin{align}\label{eq:contBoundM}
        \vert{a_{\epsilon}(\Vector{v}_h, E_l \Vector{w}_{h,l})}\vert
        & = \big\vert({\nabla\times \Vector{v}_h},{\nabla\times E_l \Vector{w}_{h,l}})_{{\widetilde{\Omega}_l}} - (\kappa^2+\mathrm{i}\epsilon)({\Vector{v}_h},{\Vector{w}_{h,l}})_{\widetilde{\Omega}_l} \notag\\
        & -\mathrm{i}\kappa\langle{\Vector{v}_{h,T}},{(E_l \Vector{w}_{h,l})_T}\rangle_{\partial\widetilde{\Omega}_l\cap\partial\Omega} \big\vert \notag\\
        &\lesssim \Vert{\nabla\times\Vector{v}_h}\Vert_{0,\widetilde{\Omega}_l}\Vert{\nabla\times E_l\Vector{w}_h}\Vert_{0,\widetilde{\Omega}_l} + \kappa^2\Vert{\Vector{v}_h}\Vert_{0,\widetilde{\Omega}_l}\Vert{ E_l\Vector{w}_h}\Vert_{0,\widetilde{\Omega}_l}\notag\\
        &+\kappa\Vert{\Vector{v}_{h,T}}\Vert_{0,\partial\widetilde{\Omega}_l\cap\partial\Omega}\Vert{ (E_l\Vector{w}_h)_T}\Vert_{0,\partial\widetilde{\Omega}_l\cap\partial\Omega}\notag\\
        &\lesssim \Vert{\Vector{v}_h}\Vert_{\imp^-,\kappa,\widetilde{\Omega}_l} \Vert{E_l \Vector{w}_{h,l}}\Vert_{\imp^-,\kappa,\widetilde{\Omega}_l}.
    \end{align}
    It's easy to check that $E_l \Vector{w}_{h,l} =\Vector{w}_{h,l}$ on $\Omega_l$ from the definition of $E_l$. Then there holds
    \begin{align}\label{eq:elwhlSplit}
        \Vert{E_l \Vector{w}_{h,l}}\Vert_{\imp^-,\kappa,\widetilde{\Omega}_l}^2
        &= \Vert{E_l \Vector{w}_{h,l}}\Vert_{\imp^-,\kappa,\Omega_l}^2 + \Vert{E_l \Vector{w}_{h,l}}\Vert_{\imp^-,\kappa,\widetilde{\Omega}_l\backslash\Omega_l}^2\cr
        &= \Vert{\Vector{w}_{h,l}}\Vert_{\imp^-,\kappa,\Omega_l}^2 + \Vert{E_l \Vector{w}_{h,l}}\Vert_{\imp^-,\kappa,\widetilde{\Omega}_l\backslash\Omega_l}^2\cr
        &\leq \Vert{\Vector{w}_{h,l}}\Vert_{\imp,\kappa,\Omega_l}^2 + \Vert{E_l \Vector{w}_{h,l}}\Vert_{\imp^-,\kappa,\widetilde{\Omega}_l\backslash\Omega_l}^2.
    \end{align}
    Next we estimate $\Vert{E_l \Vector{w}_{h,l}}\Vert_{\imp^-,\kappa,\widetilde{\Omega}_l\backslash\Omega_l}^2$. Let $\mathcal{N}_h(\widetilde{\Omega}_l\backslash\Omega_l)$ be the set of all edges,faces and
    elements on $\widetilde{\Omega}_l\backslash\Omega_l$. We additionally define ${\mathcal{N}}^{\partial}_{l,1}\subseteq \mathcal{N}_h(\widetilde{\Omega}_l\backslash\Omega_l)$ and ${\mathcal{N}}^{\partial}_{l,2}\subseteq \mathcal{N}_h(\widetilde{\Omega}_{l}\backslash\Omega_l)$ as the set consists of all edges/faces/elements touching $\partial\Omega_l\backslash\partial\Omega$ and that consists of all edges/faces/elements touching $(\partial\widetilde{\Omega}_l\backslash\partial\Omega_l)\cap\partial\Omega$ respectively. Using the inverse estimate of the finite element functions and \eqref{eq:discreteL2Norm}, we have
    \begin{align}\label{eq:whlDisNormM}
        &\quad\Vert{\nabla\times E_l\Vector{w}_{h,l}}\Vert_{0,\widetilde{\Omega}_l\backslash\Omega_l}^2 + \kappa^2\Vert{E_l\Vector{w}_{h,l}}\Vert_{0,\widetilde{\Omega}_l\backslash\Omega_l}^2\cr
        &\lesssim(h^{-2}+\kappa^2)\Vert{ E_l\Vector{w}_{h,l}}\Vert_{0,\widetilde{\Omega}_l\backslash\Omega_l}^2\cr
        &\sim (h^{-2}+\kappa^2) h \sum_{\xi\in\mathcal{N}_h(\widetilde{\Omega}_l\backslash\Omega_l)}\sum_{i=1}^{n_\xi}\vert M_{\xi}^i(E_l \Vector{w}_{h,l})\vert^2\cr
        &\sim (h^{-2}+\kappa^2) h\sum_{\xi\in\mathcal{N}^\partial_{l,1}}\sum_{i=1}^{n_\xi}\vert M_{\xi}^i(\Vector{w}_{h,l})\vert^2\cr
        &\lesssim \big[(\kappa h)^{-1}+{\kappa h}\big]\Vert{\Vector{w}_{h,l}}\Vert_{\imp,\kappa,\Omega_l}^2,
    \end{align}
    where the last inequality follows by applying \eqref{eq:discreteL2Norm} on $\partial\Omega_l\backslash\partial\Omega$. Furthermore, for the estimate of $\Vert{(E_l \Vector{w}_{h,l})_T}\Vert_{0,(\partial\widetilde{\Omega}_l\backslash\partial\Omega_l)\cap\partial\Omega}^2$, by applying \eqref{eq:discreteL2Norm} on $(\partial\widetilde{\Omega}_l\backslash\partial\Omega_l)\cap\partial\Omega$, we can obtain
    \begin{align}\label{eq:whlBdrDisNormM}
        &\quad\kappa\Vert{(E_l \Vector{w}_{h,l})_T}\Vert_{0,(\partial\widetilde{\Omega}_l\backslash\partial\Omega_l)\cap\partial\Omega}^2\cr
        &\sim \kappa\sum_{\xi\in\mathcal{N}^\partial_{l,2}}\sum_{i=1}^{n_\xi}\vert{M_{\xi}^i(E_l \Vector{w}_{h,l})}\vert^2 = \kappa\sum_{\xi\in\mathcal{N}^\partial_{l,1}\cap\mathcal{N}^\partial_{l,2}}\sum_{i=1}^{n_\xi}\vert{M_{\xi}^i(\Vector{w}_{h,l})}\vert^2\cr
        &\leq \kappa\sum_{\xi\in\mathcal{N}^\partial_{l,1}}\sum_{i=1}^{n_\xi}\vert{M_{\xi}^i(\Vector{w}_{h,l})}\vert^2\lesssim \Vert{\Vector{w}_{h,l}}\Vert_{\imp,\kappa,\Omega_l}^2.
    \end{align}
    Combining \eqref{eq:whlDisNormM} and \eqref{eq:whlBdrDisNormM} yields
    \begin{align*}
        \Vert{E_l \Vector{w}_{h,l}}\Vert_{\imp^-,\kappa,\widetilde{\Omega}_l\backslash\Omega_l}^2
        &= \Vert{\nabla\times E_l\Vector{w}_{h,l}}\Vert_{0,\widetilde{\Omega}_l\backslash\Omega_l}^2 + \kappa^2\Vert{E_l\Vector{w}_{h,l}}\Vert_{0,\widetilde{\Omega}_l\backslash\Omega_l}^2\\
        &+\kappa\Vert{(E_l \Vector{w}_{h,l})_T}\Vert_{0,(\partial\widetilde{\Omega}_l\backslash\partial\Omega_l)\cap\partial\Omega}^2\\
        &\lesssim \big[1+\kappa h+(\kappa h)^{-1}\big]\Vert{\Vector{w}_{h,l}}\Vert_{\imp,\kappa,\Omega_l}^2.
    \end{align*}
    Inserting the above equation into \eqref{eq:elwhlSplit} and noting that $\kappa h\lesssim 1 \lesssim (\kappa h)^{-1}$, we have
    \begin{align*}
        \Vert{E_l \Vector{w}_{h,l}}\Vert_{\imp^-,\kappa,\widetilde{\Omega}_l}^2
        &\lesssim \big[1+\kappa h+(\kappa h)^{-1}\big]\Vert{\Vector{w}_{h,l}}\Vert_{\imp,\kappa,\Omega_l}^2\\
        &\lesssim (\kappa h)^{-1}\Vert{\Vector{w}_{h,l}}\Vert_{\imp,\kappa,\Omega_l}^2.
    \end{align*}
    Combining this equation with \eqref{eq:contBoundM}, we finally get
    \begin{equation}\label{eq:item1BoundM}
        \vert{a_{\epsilon}(\Vector{v}_h, E_l \Vector{w}^{\partial}_{h,l})}\vert\lesssim (\kappa h)^{-\frac{1}{2}}\Vert{\Vector{v}_h}\Vert_{\imp^-,\kappa,\widetilde{\Omega}_l}\Vert{\Vector{w}_{h,l}}\Vert_{\imp,\kappa,\Omega_l}.
    \end{equation}

    For the estimate of the second item of the right hand side of \eqref{eq:reformFlM}, from the continuity of $a_{\epsilon,l}(\cdot,\cdot)$, we can obtain
    \begin{align}\label{eq:contBound2M}
        \big|a_{\epsilon,l}(R_l \Vector{v}_h, \Vector{w}_{h,l})\big|\lesssim \Vert{\Vector{v}_h}\Vert_{\imp,\kappa,\Omega_l}\Vert{\Vector{w}_{h,l}}\Vert_{\imp,\kappa,\Omega_l},
    \end{align}
    where
    \begin{equation*}
        \Vert{\Vector{v}_h}\Vert_{\imp,\kappa,\Omega_l}^2 = \Vert{\Vector{v}_h}\Vert_{\imp^-,\kappa,\Omega_l}^2 + \kappa\Vert{\Vector{v}_{h,T}}\Vert_{0,\partial\Omega_l\backslash\partial\Omega}^2.
    \end{equation*}
    Let $\mathcal{N}^\partial_{l,3}$ consist of all edges/faces/elements touching $\partial\Omega_l\backslash\partial\Omega$. Then there holds
    \begin{align*}
        &\quad\kappa\Vert{\Vector{v}_{h,T}}\Vert_{\partial\Omega_l\backslash\partial\Omega}^2\\
        &\lesssim\kappa\sum_{\xi\in\mathcal{N}^\partial_{l,3}}\sum_{i=1}^{n_\xi}\vert{M_{\xi}^i(\Vector{v}_h)}\vert^2\leq \kappa\sum_{\xi\in\mathcal{N}(\Omega_l)}\sum_{i=1}^{n_\xi}\vert{M_{\xi}^i(\Vector{v}_h)}\vert^2\\
        &\lesssim \frac{\kappa}{h}\Vert{\Vector{v}_h}\Vert_{0,\Omega_l}^2\lesssim (\kappa h)^{-1}\Vert{\Vector{v}_h}\Vert_{\imp^-,\kappa,\Omega_l}^2\\
        &\lesssim (\kappa h)^{-1}\Vert{\Vector{v}_h}\Vert_{\imp^-,\kappa,\widetilde{\Omega}_l}^2.
    \end{align*}
    Inserting the above equation into \eqref{eq:contBound2M}, we can immediately get
    \begin{equation}\label{eq:item2BoundM}
        \big|a_{\epsilon,l}(R_l \Vector{v}_h, \Vector{w}_{h,l})\big|\lesssim(\kappa h)^{-\frac{1}{2}} \Vert{\Vector{v}_h}\Vert_{\imp^-,\kappa,\widetilde{\Omega}_l}\Vert{\Vector{w}_{h,l}}\Vert_{\imp,\kappa,\Omega_l}.
    \end{equation}

    Combining \eqref{eq:item1BoundM} and \eqref{eq:item2BoundM} yields
    \begin{equation*}
        \vert{\Vector{F}_l(\Vector{w}_{h,l})}\vert\lesssim (\kappa h)^{-\frac{1}{2}}\Vert{\Vector{v}_h}\Vert_{\imp^-,\kappa,\widetilde{\Omega}_l}\Vert{\Vector{w}_{h,l}}\Vert_{\imp,\kappa,\Omega_l}.
    \end{equation*}
    Then from Corollary \ref{coro:discreteStable}, we can finally derive
    \begin{equation*}
        \Vert{\Vector{v}_{h,l}^\partial}\Vert_{\imp,\kappa,\Omega_l} \lesssim \min\bigg\{ {1+\kappa H_l,\frac{\kappa^2}{\epsilon}}\bigg\} (\kappa h)^{-\frac{1}{2}} \Vert{\Vector{v}_h}\Vert_{\imp^-,\kappa,\widetilde{\Omega}_l},
    \end{equation*}
    which is exactly \eqref{eq:vhpBoundM}.
\end{proof}

\subsection{Convergence analysis of the preconditioner}
Recall that in \eqref{eq:globalProjReformM}, $P^{\rho}_{\epsilon} \Vector{v}_h$ is reformulated as
$$P^{\rho}_{\epsilon} \Vector{v}_h = (I-P^{\rho}_{0})\Vector{v}_{h}^{\partial} + \Vector{v}_h.$$
The following lemma gives a bound of the norm of $(I-P^{\rho}_{0})\Vector{v}_{h}^{\partial}$.
\begin{lemma}
    Under Assumption \ref{assump:finiteOverlap}, there holds
    \begin{equation}\label{eq:vhpProjErrorM}
        \Vert{(I-P^{\rho}_{0})\Vector{v}_h^\partial}\Vert_{\imp,\kappa,\Omega} \lesssim \frac{\rho\kappa^2}{\epsilon}\Big(\sum_{l=1}^{N}\Vert{\Vector{v}_{h,l}^\partial}\Vert_{\imp,\kappa,\Omega_l}^2\Big)^{\frac{1}{2}}.
    \end{equation}
\end{lemma}
\begin{proof}
    Let $\Vector{z}_h\in V^{\rho}_{h,0}(\Omega)$ be defined as in Lemma \ref{lemma:vhpApproxM}. Then by the continuity and the coercivity of $a_{\epsilon}(\cdot,\cdot)$, i.e. \eqref{eq:continuousM} and \eqref{eq:coerciveM}, and the definition of $P_{\epsilon,0}^\rho$, we have

    \begin{align*}
        \Vert (I-{P^{\rho}_{\epsilon,0}})\Vector{v}_h^\partial \Vert_{\imp,\kappa,\Omega}^2
        & \lesssim \frac{\kappa^2}{\epsilon} \left\vert a_{\epsilon}((I-{P^{\rho}_{\epsilon,0}})\Vector{v}_h^\partial,(I-{P^{\rho}_{\epsilon,0}})\Vector{v}_h^\partial)\right\vert\\
        &= \frac{\kappa^2}{\epsilon} \left\vert a_{\epsilon}((I-{P^{\rho}_{\epsilon,0}})\Vector{v}_h^\partial,\Vector{v}_h^\partial)\right\vert\\
        &= \frac{\kappa^2}{\epsilon} \left\vert a_{\epsilon}((I-{P^{\rho}_{\epsilon,0}})\Vector{v}_h^\partial,\Vector{v}_h^\partial-\Vector{z}_h)\right\vert\\
        & \lesssim \frac{\kappa^2}{\epsilon} \Vert (I-{P^{\rho}_{\epsilon,0}})\Vector{v}_h^\partial \Vert_{\imp,\kappa,\Omega} \Vert \Vector{v}_h^\partial-\Vector{z}_h \Vert_{\imp,\kappa,\Omega},
    \end{align*}
    namely,
    \begin{equation*}
        \Vert (I-{P^{\rho}_{\epsilon,0}})\Vector{v}_h^\partial \Vert_{\imp,\kappa,\Omega}
        \lesssim \frac{\kappa^2}{\epsilon} \Vert \Vector{v}_h^\partial-\Vector{z}_h \Vert_{\imp,\kappa,\Omega}
        \lesssim \frac{\rho\kappa^2}{\epsilon} \left(\sum_{l=1}^{N}\Vert \Vector{v}_{h,l}^\partial \Vert_{\imp,\kappa,\Omega_l}^2\right)^{\frac{1}{2}},
    \end{equation*}
    where the last inequality holds from Lemma \ref{lemma:vhpApproxM}.
\end{proof}

Inserting \eqref{eq:vhpBoundM} into \eqref{eq:vhpProjErrorM} yields further results, leading to the main convergence result. Before giving the final result, we make an assumption on $\epsilon$ and $H$.
\begin{assumption}\label{assump:epdh}
    $\epsilon$ and $H$ satisfy
    \begin{equation}
        \epsilon\sim\kappa^{1+\alpha}\quad \text{and}\quad H\sim\kappa^{-\beta},
    \end{equation}
    where $\alpha$ and $\beta$ satisfy $0\leq \alpha\leq\beta\leq 1$.
\end{assumption}
\begin{theorem}\label{thm:theoryBigepM}
    Let Assumption \ref{assump:meshsize}, Assumption \ref{assump:finiteOverlap} and Assumption \ref{assump:epdh} be satisfied. Define $\sigma = 2-(\alpha+\beta)+\gamma/2$. For any $\tau\in (0, 1)$, there exists a constant $C_0$ independent of $\kappa,\epsilon,h,d,\delta$ (but $C_0$ depends on $\tau$)  such that if $\rho$ satisfies $\rho\kappa^{\sigma}\leq C_0$, the global projection operator $P^{\rho}_{\epsilon}=B^{-1}_{\epsilon}A_{\epsilon}$ processes the following spectral properties:
    \begin{equation}\label{eq:p0Spectrum1M}
        \Vert{P^{\rho}_{\epsilon} \Vector{v}_h}\Vert_{\imp,\kappa,\Omega}\leq (1+\tau)\Vert{\Vector{v}_h}\Vert_{\imp,\kappa,\Omega},\quad\forall \Vector{v}_h\in V_h(\Omega)
    \end{equation}
    and
    \begin{equation}\label{eq:p0Spectrum2M}
        \vert{({\Vector{v}_h},{P^{\rho}_{\epsilon} \Vector{v}_h})_{\imp,\kappa,\Omega}}\vert \geq(1-\tau) \Vert{\Vector{v}_h}\Vert_{\imp,\kappa,\Omega}^2\quad \forall \Vector{v}_h\in V_h(\Omega).
    \end{equation}
\end{theorem}
\begin{proof}
    Under Assumption \ref{assump:epdh}, $\kappa H_l\sim\kappa^{1-\beta}$ and $\kappa^2/\epsilon=\kappa^{1-\alpha}\gtrsim\kappa^{1-\beta}$. Then the factor of the right hand side of \eqref{eq:vhpBoundM} writes
    \begin{equation*}
        \min\bigg\{ 1+\kappa H_l,\frac{\kappa^2}{\epsilon}\bigg\} (\kappa h)^{-\frac{1}{2}}\sim \kappa^{1-\beta+\frac{\gamma}{2}},
    \end{equation*}
    namely
    \begin{equation*}
        \Vert{\Vector{v}_{h,l}^\partial}\Vert_{\imp,\kappa,\Omega_l}\lesssim \kappa^{1-\beta+\frac{\gamma}{2}}\Vert{\Vector{v}_h}\Vert_{\imp^-,\kappa,\widetilde{\Omega}_l}.
    \end{equation*}
    Inserting the above equation into \eqref{eq:vhpProjErrorM} yields
    \begin{align*}
        \Vert{(I-P^{\rho}_{0})\Vector{v}_h^\partial}\Vert_{\imp,\kappa,\Omega}
        &\lesssim \rho\kappa^{2-\alpha+\gamma/2-\beta}\bigg(\sum_{l=1}^{N}\Vert{\Vector{v}_{h}}\Vert_{\imp^-,\kappa,\Omega_l}^2\bigg)^{\frac{1}{2}}\\
        &\lesssim \rho\kappa^{\sigma}\Vert{\Vector{v}_{h}}\Vert_{\imp,\kappa,\Omega},
    \end{align*}
    where the last inequality follows from Lemma \ref{lemma:rlStableM} and the definition of $\sigma$. Hence for any $\tau\in (0, 1)$, there exists a constant $C_0$ independent of $\kappa,\epsilon,h,d,\delta$
    such that when $\rho\kappa^{\sigma}\leq C_0$, we have
    \begin{equation*}
        \Vert{(I-P^{\rho}_{0})\Vector{v}_h^\partial}\Vert_{\imp,\kappa,\Omega}\leq \tau\Vert{\Vector{v}_h}\Vert_{\imp,\kappa,\Omega}.
    \end{equation*}
    Therefore, from \eqref{eq:globalProjReformM} we can obtain
    \begin{align*}
        \Vert{P^{\rho}_{\epsilon}\Vector{v}_h}\Vert_{\imp,\kappa,\Omega}
        &\leq \Vert{(I-P^{\rho}_{0})\Vector{v}_h^\partial}\Vert_{\imp,\kappa,\Omega} + \Vert{\Vector{v}_h}\Vert_{\imp,\kappa,\Omega}\\
        &\leq (1+\tau)\Vert{\Vector{v}_h}\Vert_{\imp,\kappa,\Omega}.
    \end{align*}
    On the other hand, there holds
    \begin{align*}
        \big\vert({{\Vector{v}_h},{P^\rho_\epsilon\Vector{v}_h}})_{\imp,\kappa,\Omega}\big\vert
        &= \big\vert{\Vert{\Vector{v}_h}\Vert_{\imp,\kappa,\Omega}^2+ ({\Vector{v}_h},{(I-P_\epsilon^\rho)\Vector{v}_h^\partial})_{\imp,\kappa,\Omega}}\big\vert\\
        &\geq \Vert{\Vector{v}_h}\Vert_{\imp,\kappa,\Omega}^2 - \Vert{\Vector{v}_h}\Vert_{\imp,\kappa,\Omega}\Vert{(I-P_\epsilon^\rho)\Vector{v}_h^\partial}\Vert_{\imp,\kappa,\Omega}\\
        &\geq (1-\tau)\Vert{\Vector{v}_h}\Vert_{\imp,\kappa,\Omega}^2.
    \end{align*}
\end{proof}

To conclude this section, we will derive the convergence result of GMRES. Let $\mathcal{A}_\epsilon$ be the stiffness matrix of the system \eqref{eq:discreteMVF} and $\mathcal{B}_\epsilon^{-1}$ be the matrix form of the preconditioner. It's easy to check that the matrix form of $P_\epsilon^\rho$ is exactly $\mathcal{B}_\epsilon^{-1} \mathcal{A}_\epsilon$. Further, denote by $\mathcal{M}$ the matrix induced by $(\cdot,\cdot)_{\imp,\kappa,\Omega}$ on $V_h(\Omega)$. Denote by $\Vert\cdot\Vert_\mathcal{M}$ and $(\cdot,\cdot)_\mathcal{M}$ the matrix norm and corresponding inner product induced by $\mathcal{M}$ respectively, i.e.,
\begin{equation*}
    (\Vector{v},\Vector{w})_\mathcal{M}\coloneqq \Vector{w}^t\mathcal{M}\Vector{v},\quad \forall \Vector{v},\Vector{w}\in \mathbb{C}^n
\end{equation*}
and
\begin{equation*}
    \Vert\Vector{v}\Vert_\mathcal{M} \coloneqq  \sqrt{(\Vector{v},\Vector{v})_\mathcal{M}},\quad \forall \Vector{v}\in \mathbb{C}^n,
\end{equation*}
where $n=\mathop{\mathrm{dim}}V_h(\Omega)$.

Then from Theorem \ref{thm:theoryBigepM} and the basic convergence result of GMRES \cite{beckermann2005some}, we can derive the following results. We omit the detail proofs since they are the same as those in \cite[Section 5]{bonazzoli2019domain}.
\begin{theorem}\label{thm:spectral}
    Under the assumptions made in Theorem \ref{thm:theoryBigepM}, there holds
    \begin{equation*}
        \Vert \mathcal{B}_\epsilon^{-1} \mathcal{A}_\epsilon\Vert_{\mathcal{M}} \leq 1+\tau
    \end{equation*}
    and
    \begin{equation*}
        \Big\vert\big(\Vector{v},\mathcal{B}_\epsilon^{-1} \mathcal{A} \Vector{v}\big)_{\mathcal{M}}\Big\vert \geq(1-\tau) \Vert \Vector{v}\Vert_{\mathcal{M}}^2,\quad \forall \Vector{v}\in \mathbb{C}^n.
    \end{equation*}
\end{theorem}

\begin{corollary}[GMRES convergence for left-preconditioning]\label{coro:gmres}
    Let the assumptions made in Theorem \ref{thm:spectral} be satisfied. Consider the weighted GMRES applied to $\mathcal{B}_\epsilon^{-1} \mathcal{A}_\epsilon$, where the residual is minimized by $\Vert\cdot\Vert_\mathcal{M}$. Then the $m$-th iterate of GMRES, $\Vector{r}^m$, satisfies
    \begin{equation*}
        \frac{\Vert \Vector{r}^m\Vert_\mathcal{M}}{\Vert\Vector{r}^0\Vert_\mathcal{M}}\lesssim \bigg(\frac{2\sqrt{\tau}}{1+\tau}\bigg)^m.
    \end{equation*}
\end{corollary}

\begin{remark}
    Note that Corollary \ref{coro:gmres} only provides the convergence result of weighted GMRES. In \cite[Corollary 5.8]{gong2021domain}, the authors provide the convergence result of standard GMRES with Euclidean inner product, showing that the extra iterations should be paid is proportional to $\log_2(\kappa)/p$. However, there is very little difference between the results with weighted GMRES and standard GMRES from the numerical results (see \cite{graham2017domain,graham2020domain,gong2021domain,bonazzoli2019domain}). Hence we only test standard GMRES in the numerical experiments in Section \ref{sec:experiments}.
\end{remark}

\section{An economical preconditioner}\label{sec:economical}
When constructing the spectral coarse space introduced in Section \ref{sec:preconditioner}, the main cost comes from solving a series of eigenproblems. In this section, we introduce an economical coarse space to avoid solving eigenproblems without decreasing the efficiency of the preconditioner. In \cite{hu2024novel}, an economical coarse space is constructed via a coarse finite element space on $[0,1]$ and a special map from $L^2([0,1])$ to $L^2(\partial\Omega_l)$.
However such approach is not practical in the current situation of three dimensions, where the boundary of a subdomain consists of many irregular polygons and a desired map is hard to build. In this section, we will
construct an economical coarse space by the spherical harmonic functions.

From the convergence analysis in Section \ref{sec:convergence}, we know that functions in the local discrete Maxwell-harmonic space $V_h^\partial(\Omega_l)$ can be well approximated by functions in the local coarse space if the tolerance $\rho$ is chosen properly (see Lemma \ref{lemma:eigenEstimate}). Besides, a local discrete Maxwell-harmonic function $\Vector{v}_h\in V_h^\partial(\Omega_l)$ is fully determined by an ``impedance'' data $\Vector{\lambda}_h\in V_h(\partial\Omega_l\backslash\partial\Omega)$:
\begin{equation*}
    a_{\epsilon,l}(\Vector{v}_h,\Vector{w}_h)=\langle \Vector{\lambda}_h,\Vector{w}_{h}\rangle_{\partial\Omega_l\backslash\partial\Omega}, \quad\forall \Vector{w}_h\in V_h(\Omega_l).
\end{equation*}
Therefore we can directly build a space spanned by some special functions on $\partial\Omega_l\backslash\partial\Omega$ and derive a local coarse space by applying $\mathcal{H}_{\epsilon,l}$. Such local coarse space should be able to approximate $V_h^\partial(\Omega_l)$ if the special functions on $\partial\Omega_l\backslash\partial\Omega$ are chosen properly.

We will use the vector spherical harmonics to construct the economical coarse space \cite{atkinson2012spherical,barrera1985vector}. Let $\mathbb{S}^2$ be the unit sphere in $\mathbb{R}^3$. Denote the spherical harmonics on $\mathbb{S}^2$ by $Y_j^k,0\leq|k|\leq j$. Then $\{Y_j^k\}_{0\leq|k|\leq j}$ forms an orthogonal basis of $L^2(\mathbb{S}^2)$\cite{atkinson2012spherical}. Further, based on spherical harmonics $Y_j^k$, vector spherical harmonics could be defined by
\begin{equation*}
    \Vector{Y}_{j,1}^k = Y_j^k\Vector{e}_r,\quad\Vector{Y}_{j,2}^k = \nabla_S Y_j^k,\quad\Vector{Y}_{j,3}^k = \nabla_S Y_j^k\times\Vector{e}_r,
\end{equation*}
where $\nabla_S$ is the differential operator on $\mathbb{S}^2$ and $\Vector{e}_r$ is the unit vector pointing to $(x,y,z)\in \mathbb{S}^2$. Then $\big\{\Vector{Y}_{j,1}^k,\Vector{Y}_{j,2}^k,\Vector{Y}_{j,3}^k\big\}_{0\leq \vert k\vert \leq j}$ forms an orthogonal basis of $L^2(\mathbb{S}^3)^3$ and specially, $\big\{\Vector{Y}_{j,2}^k,\Vector{Y}_{j,3}^k\big\}_{0\leq \vert k\vert \leq j}$ forms an orthogonal basis of $L^2_T(\mathbb{S}^3)$\cite{barrera1985vector}.

Assume that $\Omega_l$ is star-shaped with a ball $B_r(\Vector{y})$, then we can define a homeomorphism $\Phi:\Omega_l\rightarrow B_1(\Vector{y})$ by
\begin{equation*}
    \Phi(\Vector{x}) =
    \left\lbrace
    \begin{aligned}
        &\Vector{x}_0 + \frac{\vert\Vector{x}-\Vector{y}\vert}{\vert\widetilde{\Vector{x}}-\Vector{y}\vert^2}(\widetilde{\Vector{x}}-\Vector{y}), && \text{if }\Vector{x}\neq\Vector{y}\\
        &\Vector{0}, && \text{if }\Vector{x}=\Vector{y}.
    \end{aligned}
    \right.
\end{equation*}
where $\widetilde{\Vector{x}}$ is the intersection of the ray $\overrightarrow{\Vector{y}\Vector{x}}$ and $\partial\Omega_l$. It's clear that this map is bijective and bicontinuous from $\Omega_l$ to $B_1(\Vector{y})$ and from $\partial\Omega_l$ to $\partial B_1(\Vector{y})$ since $\Omega_l$ is star-shaped. Similar homeomorphism is used to build the stability estimate of time-harmonic Maxwell equations on polyhedral domains \cite{hiptmair2011stability}.

For a given integer $\mu\geq 0$, define
\begin{equation*}
    V^\mu(\partial B_1(\Vector{y})) \coloneqq \mathop{\rm span}\big\{\Vector{Y}^k_{2,j}(\Vector{x}-\Vector{y}),\Vector{Y}^k_{3,j}(\Vector{x}-\Vector{y}):~0\leq\vert k\vert\leq j,0\leq j\leq\mu\big\}.
\end{equation*}
With the help of $\Phi$, we can pull the functions in $V^\mu(\partial B_1(\Vector{y}))$ back to those on $\partial\Omega_l$ and define
\begin{equation*}
    V^\mu(\partial\Omega_l) \coloneqq \bigg\{\nabla\Psi^{-t}\widehat{\Vector{\lambda}}:~\widehat{\Vector{\lambda}}\in V^\mu(\partial B_1(\Vector{y}))\bigg\}.
\end{equation*}
Finally, the economical local coarse space is defined by
\begin{equation*}
    V_{h,0}^\mu(\Omega_l) \coloneqq \bigg\{\mathcal{H}_{\epsilon,l}(\Vector{\lambda}_h):~\Vector{\lambda}_h\in V_h^\mu(\partial\Omega_l\backslash\partial\Omega)\bigg\},
\end{equation*}
where $V_h^\mu(\partial\Omega_l\backslash\partial\Omega)\coloneqq \Vector{r}_h(V_\mu(\partial\Omega_l)|_{\Gamma_l})$. Further, the economical coarse space is defined as
\begin{equation*}
    V_{h,0}^\mu(\Omega)\coloneqq\bigg\{\Vector{v}_h=\sum_{l=1}^{N}D_lE_l\Vector{v}_{l}:~\Vector{v}_{l}\in V_{h,0}^\mu(\Omega_l)\bigg\}
\end{equation*}
and the economical preconditioner could be built in the same way as in Section \ref{sec:preconditioner}.

\section{Numerical experiments}\label{sec:experiments}
In this section we give some numerical results to illustrate the efficiency of the preconditioner introduced in Section \ref{sec:preconditioner} and the economical variant introduced in Section \ref{sec:economical}.

We choose $\Omega$ as the unit cube $[0,1]^3$. Throughout this section, we make use of standard $p$-order N\'ed\'elec finite element approximation based on the structured tetrahedral grid of size $h\approx\kappa^{-\frac{2p+1}{p}}$ (which can remove the pollution effect) or with a fixed number $n_{ppw}$ of grid points per wavelength (usually used in engineering). As for the domain decomposition, we make use of a uniform decomposition into $N$ cubic non-overlapping subdomains and then extend them by adding several layers of fine-mesh elements to obtain overlapping subdomains. We call the case {\em minimal overlap} if only one layer of elements is added ($\delta\sim h$). Besides, the term {\em generous overlap} refers to $\delta\sim H$. If not specially mentioned, we utilize {\em generous overlap} with $\delta\approx H/6$ in this section.

Standard GMRES with zero initial guess is used. The stopping criterion is based on a reduction of the relative residual by $10^{-6}$ and the maximum number of iterations allowed is 1000.

The implementation is mainly based on MFEM\cite{anderson2021mfem}, which is a free, lightweight, scalable C\texttt{++} library for finite element methods. Moreover, we use MUMPS\cite{amestoy2001fully}, which is integrated in PETSc\cite{balay2023petsc}, both as the subdomain solver and the coarse solver. When it comes to the construction of the spectral coarse spaces, multiple generalized eigenvalue problems are solved, and SLEPc\cite{hernandez2005slepc} is used to do these computations. The numerical experiments in this section are done on the high performance computers of State Key Laboratory of Scientific and Engineering Computing, CAS.

\subsection{Model with constant wavenumber} In this subsection, we consider the simulation of a model with constant wavenumber. $\Vector{J}$ and $\Vector{g}$ in \eqref{eq:shiftMaxwell} is
\begin{align*}
    &\Vector{J} = \nabla\times(\nabla\times \Vector{E}_{ex})-\kappa^2\Vector{E}_{ex},\\
    &\Vector{g} = (\nabla\times\Vector{E}_{ex})\times\Vector{n}-\mathrm{i}\kappa\Vector{E}_{ex,T},
\end{align*}
where
\begin{equation*}
    \Vector{E}_{ex} =
    \begin{pmatrix}
        xz\sin(\kappa y) + \mathrm{i}yz\cos(\kappa x)\\
        -z\sin(\kappa y) - \mathrm{i}z\sin(\kappa x)\\
        xy + \mathrm{i}xy
    \end{pmatrix}.
\end{equation*}

In Section \ref{sec:convergence}, the requirement on the tolerance $\rho$ ensuring convergence is proposed (see Theorem \ref{thm:theoryBigepM}). We make the first experiment to verify the theory. In this experiment we choose $\epsilon=\kappa$ which is meaningful in application and $\rho\sim\kappa^{\beta-2-\frac{\gamma}{2}}$ by Theorem \ref{thm:theoryBigepM}. As mentioned before, the mesh size $h\approx\kappa^{-\frac{2p+1}{2p}}$ where $p=1,2$, implying that $\gamma = \frac{1}{2p}$. Moreover, generous overlap is utilized. The GMRES iteration counts, dimension of coarse space and maximal size of local problems for this case are showed in Table \ref{tab:theoryCheckM}.

\begin{table}[H]
    \caption{\rm GMRES iteration counts (top of the cell), dimension of the coarse space (middle of the cell) and maximal size of local problems (bottom of the cell)}
    \label{tab:theoryCheckM}
    \centering
    \small
    \begin{tabular}{ccccccccc}
        \toprule
        & \multicolumn{4}{c}{linear FE approximation} & \multicolumn{4}{c}{quatratic FE approximation} \\
        \cmidrule(r){2-5} \cmidrule(r){6-9}
        $\kappa$ & \#dof$\backslash\beta$ & 0.8 & 0.7 & 0.6 & \#dof$\backslash\beta$ & 0.8 & 0.7 & 0.6 \\
        \midrule
        $4\pi$ & 656,235 & 32 & 13 & 8 & 542,736 & 26 & 13 & 8 \\
        & & 126,944 & 74,829 & 73138 & & 78,474 & 98,356 & 86,615\\
        & & 3,956 & 12,312 & 15721 & & 4,882 & 8,532 & 13,610\\
        \\
        $5\pi$ & 1,703,078 & 14 & 7 & 5 & 1,161,074 & 10 & 8 & 7 \\
        & & 270,482 & 228,402 & 148,753 & & 288,238 & 233,192 & 139,027\\
        & & 9,425 & 15,721 & 47,587 & & 8,532 & 13,610 & 49,198\\
        \\
        $6\pi$ & 3,920,338 & 6 & 5 & 4 & 2,299,986 & 10 & 6 & 5 \\
        & & 640,464 & 528,684 & 351,906 &  & 573,613 & 484,908 & 348,073\\
        & & 15,721 & 33,743 & 55,712 & & 8,532 & 13,610 & 49,198\\
        \bottomrule
    \end{tabular}
\end{table}

From Table \ref{tab:theoryCheckM}, we can clearly see that the number of iterations keeps bounded as $\kappa$ increases for each $\beta$, which verifies the theory. Actually, with such choice of tolerance $\rho$, the number of iterations decreases significantly as $\kappa$ grows. This phenomenon implies that the choice of tolerance $\rho$ in the theory is not optimal, leading to a dramatic increase in the dimension of the coarse space. Besides, the results in Table \ref{tab:theoryCheckM} show that the preconditioner shows similar efficiency for both finite element approximations. We will utilize quadratic approximation in the remaining numerical experiments since it exhibits good performance for solving high-frequency problems.

Constructing the spectral coarse space is costly since a series of generalized eigenproblems should be solved. In Section \ref{sec:economical}, we proposed an economical variant to avoid solving eigenproblems, which significantly saves the computational cost. The next experiment is designed to show the efficiency of the economical preconditioner. The key parameter of the economical preconditioner is $\mu$ (see Section \ref{sec:economical} for details). Intuitively, when $H$ and $\kappa$ grow, the parameter $\mu$ should increase so that each local coarse space contains more basis functions. We choose $\mu=\kappa^{1-\beta/2}$ and list the results in Table \ref{tab:ecoM}.

\begin{table}[H]
    \caption{\rm GMRES iteration counts (top of the cell), dimension of the coarse space (middle of the cell) and maximal size of local problems (bottom of the cell) for the economical preconditioner with $\mu=\kappa^{1-\beta/2}$}
    \label{tab:ecoM}
    \small
    \begin{center}
        \begin{tabular}{cccccccc}
            \toprule
            & & \multicolumn{3}{c}{$\epsilon=\kappa$} & \multicolumn{3}{c}{$\epsilon=0$}\\
            \cmidrule(r){3-5} \cmidrule(r){6-8}
            $\kappa$ & \#dof$\backslash\beta$ & 0.7 & 0.6 & 0.5 & 0.7 & 0.6 & 0.5\\
            \midrule
            $6\pi$ & 2,299,986 & 20 & 11 & 10 & 21 & 11 & 10\\
            & & 64,512 & 34,560 & 12,672 & 64,512 & 34,560 & 12,672\\
            & & 13,610 & 49,198 & 103,612 & 13,610 & 49,198 & 103,612\\
            \\
            $8\pi$ & 6,767,824 & 26 & 10 & 9 & 27 & 12 & 9\\
            & & 160,000 & 82,320 & 35,750 & 160,000 & 82,320 & 35,750\\
            & & 20,344 & 64,464 & 213,396 & 20,344 & 64,464 & 213,396\\
            \\
            $10\pi$ & 15,563,236 & 14 & 12 & 9 & 15 & 13 & 10\\
            & & 263,538 & 146,432 & 84,240 & 263,538 & 146,432 & 84,240\\
            & & 49,198 & 103,612 & 252,422 & 49,198 & 103,612 & 252,422\\
            \bottomrule
        \end{tabular}
    \end{center}
\end{table}

It can be seen in Table \ref{tab:ecoM} that there is a mild increase in the iteration counts as $\kappa$ grows, especially when $\beta=0.5, 0.6$. When $\beta=0.7$, the iteration counts increase gradually as $\kappa$ grows from $4\pi$ to $8\pi$ and decrease steeply as $\kappa$ changes to $10\pi$. This is because in the actual computation the ratio $\delta/H$ exhibits periodic variations. In each cycle, the ratio $\delta/H$ shows a gradually increase followed by a sudden increase as $\kappa$ grows. Further, the results for the case with $\epsilon=0$ indicate that the economical preconditioner performs well in solving the pure problems.

At the end of this section, we compare the performance of several preconditioners. For convenience, we use ECS to denote the economical preconditioner, WASI to denote the one-level preconditioner introduced in Subsection \ref{subsec:1level}, and GCS to denote the two-level preconditioner with grid coarse space (see \cite{bonazzoli2019domain} for example). The results are listed in Table \ref{tab:comp1M}.

It's well known that GCS processes robustness only if $\kappa H = O(1)$ \cite{bonazzoli2019domain}. The results in Table \ref{tab:comp1M} show that if this condition is not satisfied, GCS is less effective and even the introduction of grid coarse space doesn't gain any benefit. The iteration counts of ECS are significantly less than those of WASI and GCS. Moreover, there is at most a slight increase in the iteration counts for ECS with both generous overlap and minimal overlap.

\begin{table}[H]
    \caption{\rm GMRES iteration counts
    for three preconditioners}
    \label{tab:comp1M}
    \small
    \begin{center}
        \begin{tabular}{cc*{6}{p{18pt}<{\centering}}}
            \toprule
            & & \multicolumn{3}{c}{ Generous overlap} & \multicolumn{3}{c}{ Minimal overlap}\\
            \cmidrule(r){3-5} \cmidrule(r){6-8}
            & $\kappa\backslash\beta$ & 0.7 & 0.6 & 0.5 & 0.7 & 0.6 & 0.5\\
            \midrule
            \multirow{3}{*}{ECS}& $6\pi$ & 21 & 11 & 10 & 20 & 18 & 18 \\
            & $8\pi$ & 27 & 12 & 9 & 26 & 20 & 20 \\
            & $10\pi$ & 15 & 13 & 10 & 26 & 23 & 22 \\
            \\
            \multirow{3}{*}{WASI}& $6\pi$ & 77 & 40 & 28 & 77 & 62 & 43 \\
            & $8\pi$ & 98 & 46 & 29 & 98 & 73 & 54 \\
            & $10\pi$ & 68 & 53 & 35 & 110 & 86 & 66 \\
            \\
            \multirow{3}{*}{GCS}& $6\pi$ & 67 & 36 & 31 & 67 & 62 & 43 \\
            & $8\pi$ & 72 & 21 & 35 & 72 & 73 & 54 \\
            & $10\pi$ & 68 & 75 & 58 & 110 & 86 & 66 \\
            \bottomrule
        \end{tabular}
    \end{center}
\end{table}

\subsection{Electric dipole in layered media}
In this subsection, we consider the computation of the field due to an electric dipole in layered media, which is used to illustrate the performance of several plane wave methods \cite{hu2014plane,huttunen2007solving}. Specifically, let $\Omega=[0,1]^3$ and $\Vector{x}_0=(0.5,0.5,0.8)$ be the location of the electric dipole. Then the model writes
\begin{equation}\label{eq:dipole}
    \left\lbrace
    \begin{aligned}
         & \nabla\times(\nabla\times\Vector{E}) -\kappa^2\epsilon_r \Vector{E}=\mathrm{i}\delta_{\Vector{x}_0}\Vector{a}, & & \text{in }\Omega, \\
         & (\nabla\times\Vector{E})\times\Vector{n} -\mathrm{i}\kappa\sqrt{\mathop{\rm Re}\epsilon_r}\Vector{E}_T=0, & & \text{on }\partial\Omega,
    \end{aligned} \right.
\end{equation}
where the relative permittivity $\epsilon_r$ is complex, $\Vector{a} = (1,0,0)$ and $\delta_{\Vector{x}_0}$ is the Dirac delta function with respect to $\Vector{x}_0$. The distribution of $\epsilon_r$ is as follows:
\begin{equation*}
    \epsilon_r =
    \left\lbrace
        \begin{aligned}
            &1.0, & &0.5\leq z\leq 1,\\
            &11.5+10^{-4}\mathrm{i}, & & 0.2\leq z < 0.5,\\
            &2.1 + 10^{-5}\mathrm{i}, & &  0.025\leq z < 0.2,\\
            &1.0 + 10^{7}\mathrm{i}, & & 0\leq z < 0.025,
        \end{aligned}
    \right.
\end{equation*}
where the different layers correspond to different materials.

In the current experiment, the mesh size $h$ is chosen such that $h\approx \lambda/8$ where $\lambda = 2\pi/\kappa$ is the wavelength, implying a fixed number of points per wavelength which is usually considered in practice. As for the domain decomposition, we still use a uniform decomposition into cubes with minimal overlap. We list the results in Table \ref{tab:dipole}, where \#dof denotes the number of degrees of freedom, $N$ denotes the number of subdomains and $N_{\rm iter}$ denotes the GMRES iteration counts.

\begin{table}[H]
    \caption{\rm GMRES iteration counts, dimension of the coarse space and the maximal size of local problems for the economical preconditioner}
    \label{tab:dipole}
    \small
    \begin{center}
        \begin{tabular}{cccccccc}
            \toprule
            $\kappa$ & \#dof & $N$ & $\mu$ & $N_{\rm iter}$ & $\mathop{\rm dim}V_{h,0}^\mu(\Omega)$ & $\max_l \mathop{\rm dim}V_h(\Omega_l)$\\
            \midrule
     \multirow{3}{*}{$10\pi$} & \multirow{3}{*}{1,953,462} & 216 & 12 & 26 & 72,576 & 23,432\\
            & & 343 & 11 & 26 & 98,098 & 16,980\\
            & & 512 & 11 & 27 & 146,432 & 13,562\\
            \\
      \multirow{3}{*}{$15\pi$} & \multirow{3}{*}{6,342,048} & 216 & 18 & 29 & 155,520 & 59,596\\
            & & 343 & 17 & 32 & 221,578 & 41,378\\
            & & 512 & 15 & 29 & 261,120 & 30,516\\
            \\
        \multirow{3}{*}{$20\pi$} & \multirow{3}{*}{14,862,194} & 216 & 25 & 30 & 291,600 & 118,256\\
            & & 343 & 23 & 30 & 394,450 & 81,410\\
            & & 512 & 20 & 31 & 450,560 & 58,948\\
            \bottomrule
        \end{tabular}
    \end{center}
\end{table}

It can be seen from Table \ref{tab:dipole} that, with such choice of parameters, the iteration counts increase slightly as $\kappa$ grows, indicating that the economical preconditioner still processes robustness with respect to $\kappa$ if $\mu$ is chosen appropriately.

\appendix
\numberwithin{equation}{section}
\renewcommand{\theequation}{\thesection.\arabic{equation}}
\section{A special ``interpolation" operator for $V_h(\Omega)$}\label{sec:appendixA}

In this section, we introduce a special ```quasi-N\'ed\'elec interpolation" operator $\widetilde{\Vector{r}}_h$, which was used to define the economical weighted operator $D_l$ in subsection \ref{subsec:1level}.

Let $K\in\mathcal{T}_h$ be a tetrahedron and $R_p(K)$ be the space consisting of N\'ed\'elec functions of order $p$ (see \cite[Section 5.5]{monk2003finite} for details). For a polygon $G$, let $\mathbb{P}_p(G)$ be space consisting of all polynomials of degree less than $p+1$ on $G$ and denote by $n_p(G)$ the dimension of $\mathbb{P}_p(G)$. For a series of points $\eta^G_{i}\subset G, {0\leq i< n_p(G)}$, we say $\{\eta^G_{i}\}$ are uniquely solvable in $\mathbb{P}_p(G)$ if every $P\in \mathbb{P}_p(G)$ can be uniquely determined by $P(\eta^G_i)$. Under the above preparation and following the classical approach of Ciarlet \cite{ciarlet2002finite}, we can give the definition of quasi-N\'ed\'elec finite element on reference element $\widehat{K}$ as follows.

\begin{definition}\label{def:newdof}
    Let points $\big\{\eta^{\widehat{e}}_i,0\leq i \leq p-1\big\}$, $\big\{\eta^{\widehat{f}}_{i,j},0\leq i+j\leq p-2\big\}$ and $\big\{\eta^{\widehat{K}}_{i,j,k},0\leq i+j+k\leq p-3\big\}$ be uniquely solvable in $\mathbb{P}_{p-1}(\widehat{e})$, $\mathbb{P}_{p-2}(\widehat{f})$ and $\mathbb{P}_{p-3}(\widehat{K})$ respectively, where $\widehat{e}$ is an edge of $\widehat{K}$ and $\widehat{f}$ is a face of $\widehat{K}$. Define a finite element on $\widehat{K}$ by
    \begin{itemize}
        \item $\widehat{K}$ is a tetrahedron
        \item $P_{\widehat{K}} = R_p(\widehat{K})$
        \item $\Sigma_{\widehat{K}} = \widetilde{M}_{\widehat{e}}(\widehat{\Vector{u}})\cup \widetilde{M}_{\widehat{f}}(\widehat{\Vector{u}}) \cup \widetilde{M}_{\widehat{K}}(\widehat{\Vector{u}})$ (for every edge~
        $\widehat{e}$ and face $\widehat{f}$  of $\widehat{K}$), where
              \begin{gather}
                  \widetilde{M}_{\widehat{e}}(\widehat{\Vector{u}}) = \left\lbrace \widetilde{M}_{\widehat{e}}^i(\widehat{\Vector{u}})\coloneqq\widehat{\Vector{u}}(\eta^{\widehat{e}}_i)\cdot\widehat{\Vector{\tau}}^{\widehat{e}},0\leq i
                  \leq p-1\right\rbrace
                  \label{eq:newdof_e}\\
                  \widetilde{M}_{\widehat{f}}(\widehat{\Vector{u}}) = \left\lbrace\widetilde{M}_{\widehat{f}}^{i,j}(\widehat{\Vector{u}})\coloneqq \widehat{\Vector{u}}(\eta^{\widehat{f}}_{i,j})\cdot\widehat{\Vector{\tau}}^{\widehat{f}}_\nu,0\leq i+j\leq p-2,\nu=1,2\right \rbrace, \label{eq:newdof_f}\\
                  \widetilde{M}_{\widehat{K}}(\widehat{\Vector{u}}) = \left\lbrace\widetilde{M}_{\widehat{K}}^{i,j,k}(\widehat{\Vector{u}})\coloneqq\widehat{\Vector{u}}(\eta^{\widehat{K}}_{i,j,k})\cdot\widehat{\Vector{\tau}}^{\widehat{K}}_\nu,0\leq i+j+k\leq p-3,\nu=1,2,3\right\rbrace. \label{eq:newdof_k}
              \end{gather}
        In the above definition of degrees of freedom, $\widehat{\Vector{\tau}}^{\widehat{e}}$ is a unit vector in the direction of $\widehat{e}$, $\widehat{\Vector{\tau}}^{\widehat{f}}_\nu, \nu=1,2$ are unit vectors in the direction of arbitrary two edges of $\widehat{f}$ respectively and $\widehat{\Vector{\tau}}^{\widehat{K}}_\nu, \nu = 1,2,3$ are unit vectors in the direction of arbitrary three edges of $\widehat{K}$ respectively.
    \end{itemize}
\end{definition}

The finite element defined in Definition \ref{def:newdof} is well-posed. We give the following two lemmas without proof for simplicity.
\begin{lemma}\label{lemma:newfeConform}
    For any $\Vector{u}\in R_p(\widehat{K})$, if the degrees of freedom associated with a face $\widehat{f}$ (defined in \eqref{eq:newdof_f}) and those associated with all edges of $\widehat{f}$ (defined in \eqref{eq:newdof_e}) vanish, then $\Vector{u}\times \Vector{n} = 0$ on $\widehat{f}$.
\end{lemma}
\begin{lemma}\label{lemma:newfeWelldef}
    For any $\Vector{u}\in R_p(\widehat{K})$, if the degrees of freedom defined in \eqref{eq:newdof_e} - \eqref{eq:newdof_k} vanish, then $\Vector{u} \equiv 0$.
\end{lemma}

Combining Lemma \ref{lemma:newfeConform} and Lemma \ref{lemma:newfeWelldef}, we can derive the following theorem.
\begin{theorem}
    The finite element in Definition \ref{def:newdof} is well-possed and curl conforming.
\end{theorem}

Therefore, we could define the quasi-N\'ed\'elec finite element on arbitrary element $K$ through transformation and the quasi-N\'ed\'elec interpolation operator $\widetilde{\Vector{r}}_h$ as in \cite{monk2003finite}.

\begin{lemma}\label{lemma:newInterpWelldef}
    For any $\Vector{u}\in H(\Vector{curl};\Omega)$, assume that there exists a integer $m\geq 2$ such that $\Vector{u}\in (H^m(K))^3$ for any $K\in\mathcal{T}_h$. Then $\widetilde{\Vector{r}}_h\Vector{u}\in V_h(\Omega)$ is well-defined and there holds
    \begin{align}
         & \Vert{\widetilde{\Vector{r}}_h\Vector{u}}\Vert_{0,K}\lesssim \Vert{\Vector{u}}\Vert_{m,K}, \label{eq:newInterpL2Bound}   \\
         & \Vert{\nabla\times\widetilde{\Vector{r}}_h\Vector{u}}\Vert_{0,K}\lesssim \Vert{\Vector{u}}\Vert_{m,K}+\Vert{\Vector{u}_T}\Vert_{m-\frac{1}{2},\partial K}, \label{eq:newInterpCurlBound} \\
         & \Vert{(\widetilde{\Vector{r}}_h\Vector{u})_T}\Vert_{0,f}\lesssim \Vert{\Vector{u}_T}\Vert_{m-\frac12,f},\quad \text{where } f\subset{\partial K}\text{ is a face},\label{eq:newInterpTBound}
    \end{align}
    where the implicit constant is independent of $\Vector{u}$.
\end{lemma}
\begin{proof}
    For any $K\in \mathcal{T}_h$, assume that $\widetilde{\Vector{r}}_h\Vector{u}$ admits the following decomposition
    \begin{align*}
        \widetilde{\Vector{r}}_h\Vector{u}|_K
        &= \sum_e\sum_{0\leq i\leq p-1}\widetilde{M}_e^i(\Vector{u})\Vector{\phi}_e^i
        + \sum_f\sum_{0\leq i+j\leq p-2}\sum_{\nu=1}^2\widetilde{M}_{f,\nu}^{i,j}(\Vector{u})\Vector{\phi}_{f,\nu}^{i,j}\\
        &+ \sum_{0\leq i+j+k\leq p-3}\sum_{\nu=1}^3\widetilde{M}_{K,\nu}^{i,j,k}(\Vector{u})\Vector{\phi}_{K,\nu}^{i,j,k},
    \end{align*}
    where $\widetilde{M}_e^i$, $\widetilde{M}_{f,\nu}^{i,j}$ and $\widetilde{M}_{K,\nu}^{i,j,k}$ are defined in Definition \ref{def:newdof} and $\Vector{\phi}_e^i$, $\Vector{\phi}_{f,\nu}^{i,j}$ and $\Vector{\phi}_{K,\nu}^{i,j,k}$ denote the corresponding basis. It's easy to check that
    \begin{gather*}
        \vert{\widetilde{M}_e^i(\Vector{u})}\vert = \vert{\Vector{u}(\eta_e^i)\cdot\tau_e}\vert\leq\Vert{\Vector{u}}\Vert_{(C(K))^3}\vert{\tau_e}\vert\leq h_K\Vert{\Vector{u}}\Vert_{(C(K))^3}.
    \end{gather*}
    Similarly, we can derive
    \begin{equation*}
        \vert{\widetilde{M}_f^{i,j}(\Vector{u})}\vert \leq h_K\Vert{\Vector{u}}\Vert_{(C(K))^3},~\vert{\widetilde{M}_K^{i,j,k}(\Vector{u})}\vert \leq h_K\Vert{\Vector{u}}\Vert_{(C(K))^3}.
    \end{equation*}
    Let
    \begin{align*}
        C_{\Vector{\phi}}
        &\coloneqq \sum_e\sum_{0\leq i\leq p-1}\Vert{\Vector{\phi}_e^i}\Vert_{0,K}
        + \sum_f\sum_{0\leq i+j\leq p-2}\sum_{\nu=1}^2\Vert{\Vector{\phi}_{f,\nu}^{i,j}}\Vert_{0,K}\\
        &+ \sum_{0\leq i+j+k\leq p-3}\sum_{\nu=1}^3\Vert{\Vector{\phi}_{K,\nu}^{i,j,k}}\Vert_{0,K}<\infty.
    \end{align*}
    Then from Sobolev embedding theorem, there holds
    \begin{equation*}
        \Vert{\widetilde{\Vector{r}}_h\Vector{u}}\Vert_{0,K}\leq C_{\Vector{\phi}} h_K\Vert{\Vector{u}}\Vert_{(C(K))^3}\lesssim \Vert{\Vector{u}}\Vert_{m,K}.
    \end{equation*}
    This is \eqref{eq:newInterpL2Bound}. Moreover, $(\widetilde{\Vector{r}}_h{\Vector{u}})_T\in R_p(f)$ for any face $f\subset\partial K$. Note that $(\widetilde{\Vector{r}}_h{\Vector{u}})_T=\widetilde{\Vector{r}}_{h,f}\Vector{u}_T$, where $\widetilde{\Vector{r}}_{h,f}$ is the quasi-N\'ed\'elec interpolation operator on $R_p(f)$. Similarly, there holds
    \begin{equation*}
        \Vert{(\widetilde{\Vector{r}}_h\Vector{u})_T}\Vert_{0,f}\lesssim \Vert{\Vector{u}_T}\Vert_{(C(f))^3}\lesssim\Vert{\Vector{u}_T}\Vert_{m-\frac12,f},
    \end{equation*}
    and this is exactly \eqref{eq:newInterpTBound}. Finally, for any $ \Vector{q}\in(\mathbb{P}_{p-1}(K))^3$, we have
    \begin{align*}
        \bigg\vert{\int_K \nabla\times\widetilde{\Vector{r}}_h\Vector{u}\cdot\Vector{q}\Diff V}\bigg\vert
         & = \bigg\vert{\int_K \widetilde{\Vector{r}}_h\Vector{u}\cdot\nabla\times\Vector{q}\Diff {V} - \int_{\partial K}(\widetilde{\Vector{r}}_h\Vector{u})_T\cdot\Vector{q}_t\Diff S}\bigg\vert \\
         & \leq \Vert{\widetilde{\Vector{r}}_h\Vector{u}}\Vert_{0,K}\Vert{\nabla\times\Vector{q}}\Vert_{0,K} + \Vert{(\widetilde{\Vector{r}}_h\Vector{u})_T}\Vert_{0,\partial K}\Vert{\Vector{q}_t}\Vert_{0,\partial K}\\
         & \leq \big(\Vert{\nabla\times\Vector{q}}\Vert_{0,K}^2 + \Vert{\Vector{q}_t}\Vert_{0,\partial K}^2\big)^\frac{1}{2} \big(\Vert{\widetilde{\Vector{r}}_h\Vector{u}}\Vert_{0,K}^2 + \Vert{(\widetilde{\Vector{r}}_h\Vector{u})_T}\Vert_{0,\partial K}^2\big)^\frac{1}{2},
    \end{align*}
    where $\Vector{q}_t = \Vector{n}\times \Vector{q}$. Let
    \begin{equation*}
        C_{\Vector{q}}=\sup_{\Vector{q}\in(\mathbb{P}_{p-1}(K))^3}\frac{\big(\Vert{\nabla\times\Vector{q}}\Vert_{0,K}^2 + \Vert{\Vector{q}_t}\Vert_{0,\partial K}^2\big)^\frac{1}{2}}{\Vert{\Vector{q}}\Vert_{0,K}} < \infty,
    \end{equation*}
    then
    \begin{equation*}
        \sup_{\Vector{q}\in(\mathbb{P}_{p-1}(K))^3}\frac{\big\vert({\nabla\times\widetilde{\Vector{r}}_h\Vector{u}},{\Vector{q}})_{K}\big\vert}{\Vert{\Vector{q}}\Vert_{0,K}}\leq C_{\Vector{q}}\big(\Vert{\widetilde{\Vector{r}}_h\Vector{u}}\Vert_{0,K}^2 + \Vert{(\widetilde{\Vector{r}}_h\Vector{u})_T}\Vert_{0,\partial K}^2\big)^\frac{1}{2}.
    \end{equation*}
    Since $\nabla\times\widetilde{\Vector{r}}_h\Vector{u}\in(\mathbb{P}_{p-1}(K))^3$, then from \eqref{eq:newInterpL2Bound} and \eqref{eq:newInterpTBound}, there holds
    \begin{align*}
        \Vert{\nabla\times\widetilde{\Vector{r}}_h\Vector{u}}\Vert_{0,K}
         & \leq C_{\Vector{q}}\left(\Vert{\widetilde{\Vector{r}}_h\Vector{u}}\Vert_{0,K}^2 + \Vert{(\widetilde{\Vector{r}}_h\Vector{u})_T}\Vert_{(L^2(\partial K))^3}^2\right)^\frac{1}{2} \\
         & \lesssim \Vert{\Vector{u}}\Vert_{m,K}+\Vert{\Vector{u}_T}\Vert_{m-\frac{1}{2},\partial K}.
    \end{align*}
\end{proof}

\begin{lemma}\label{lemma:newInterpApply}
    For $K\in\mathcal{T}_h$, any $\Vector{v}_h\in V_h(K)$ admits the decomposition $\Vector{v}_h = \sum_{\eta}\widetilde{M}_\eta(\Vector{v}_h)\phi_\eta$ where $\widetilde{M}_\eta$ is a degree of freedom
    associated with the edge/face/element $\eta$ (see Definition \ref{def:newdof}). Moreover, for any $\chi\in C(K)$, there holds
    \begin{equation*}
        \widetilde{\Vector{r}}_h(\chi\Vector{v}_h) = \widetilde{\Vector{r}}_h\left(\chi\sum_{\eta}\widetilde{M}_\eta(\Vector{v}_h)\phi_\eta\right)=\sum_\eta\widetilde{M}_\eta(\Vector{v}_h)\chi(\eta)\phi_\eta.
    \end{equation*}
\end{lemma}
\begin{proof}
    The proof follows directly from Definition \ref{def:newdof}.
\end{proof}

\section{The discrete $L^2$ norm for N\'ed\'elec finite element functions}\label{sec:appendixB}
In this section, we will briefly introduce the discrete $L^2$ norm for N\'ed\'elec finite element functions. As in Appendix \ref{sec:appendixA} and refer to \cite{monk2003finite}, for a tetrahedron $K$, let $R_p(K)$ be the N\'ed\'elec function space on $K$. Meanwhile, for the reference element
$\widehat{K}$ and the affine map $\Phi_K:~\widehat{K}\rightarrow K$, the transformation induced by $\Phi_K$ is defined as
\begin{equation}\label{eq:ndfuncTrans}
    \Vector{v}(\Phi_K(\widehat{\Vector{x}})) = (B_K^t)^{-1}\widehat{\Vector{v}}(\widehat{\Vector{x}})\in R_p(K), \quad \forall\widehat{\Vector{v}}\in R_p(\widehat{K}),
\end{equation}
where $B_K = (\nabla\Phi_K)^t$ is the Jacobian of $\Phi_K$. Note that the N\'ed\'elec finite element could also be built in 2D for triangle $K$ where the resulting finite element is the so-called rotated Raviart–Thomas element \cite{monk2003finite}.

We have the following lemma.
\begin{lemma}[{\cite[Lemma 5.43]{monk2003finite}}]\label{lemma:scalingHcurl}
    Assume that $\mathcal{T}_h$ is a quasi-uniform mesh posed on $\Omega\subset\mathbb{R}^d$ with mesh size $h$. Then for any $s\geq 0$, the mapping $\widehat{\Vector{v}}\rightarrow \Vector{v}=(B_K^t)^{-1}\widehat{\Vector{v}}\circ\Phi_K^{-1}$ satisfies
    \begin{align*}
        &\vert{\widehat{\Vector{v}}}\vert_{s,\widehat{K}}\lesssim h^{s+1-\frac{d}{2}}\vert{\Vector{v}}\vert_{s,K},\\
        &\vert{\Vector{v}}\vert_{s,K}\lesssim h^{\frac{d}{2}-s-1}\vert{\widehat{\Vector{v}}}\vert_{s,\widehat{K}}.
    \end{align*}
    When $d=3$, there also holds
    \begin{align*}
        &\vert{\widehat{\nabla}\times\widehat{\Vector{v}}}\vert_{s,\widehat{K}}\lesssim h^{s+\frac{1}{2}}\vert{\nabla\times\Vector{v}}\vert_{s,K},\\
        &\vert{\nabla\times\Vector{v}}\vert_{s,K}\lesssim h^{-s-\frac{1}{2}}\vert{\widehat{\nabla}\times\widehat{\Vector{v}}}\vert_{s,\widehat{K}}.
    \end{align*}
\end{lemma}

As in the case of classical $H^1$ finite element functions, we could define the discrete $L^2$ norm of N\'ed\'elec finite element functions. Recall that we use $\mathcal{N}_h(\Omega)$ denote the set consists of all
edges/faces/elements on $\Omega$. Then for $\Vector{v}_h\in V_h(\Omega)$, we could define its discrete $L^2$ norm as
\begin{equation*}
    \Vert{\Vector{v}_h}\Vert_{0,h,\Omega}^2 \coloneqq h^{d-2}\sum_{\xi\in \mathcal{N}_h(\Omega)}\sum_{i=1}^{n_\xi}\vert M_{\xi}^i(\Vector{v}_h)\vert^2.
\end{equation*}
The following lemma can be proved by the scaling argument.
\begin{lemma}\label{lemma:discreteL2NormSimND}
    Assume that $\mathcal{T}_h$ is quasi-uniform, then for any $\Vector{v}_h\in V_h(\Omega)$ there holds
    \begin{equation*}
        \Vert{\Vector{v}_h}\Vert_{0,h,\Omega}\sim \Vert{\Vector{v}_h}\Vert_{0,\Omega}.
    \end{equation*}
\end{lemma}
\begin{proof}
    It only needs to prove element by element. Let $\mathcal{N}_h(K)\subset\mathcal{N}_h(\Omega)$ consist of all of the edges/faces of the element $K$ and $K$ itself. We have
    \begin{equation*}
        \sum_{\xi\in\mathcal{N}_h(K)}\sum_{i=1}^{n_\xi}\vert M_{\xi}^i(\Vector{v}_h)\vert^2 = \sum_{\xi\in\mathcal{N}_h(\widehat{K})}\sum_{i=1}^{n_\xi}\vert M_{\xi}^i(\widehat{\Vector{v}}_h)\vert^2\coloneqq \vertiii \widehat{\Vector{v}}_h\vertiii^2.
    \end{equation*}
    It's easy to check that $\vertiii{\cdot}\vertiii$ is a norm on $R_p(\widehat{K})$. Note that $R_p(\widehat{K})$ is finite dimensional. Then there holds
    \begin{equation*}
        \vertiii{\widehat{\Vector{v}}_h}\vertiii \sim \Vert{\widehat{\Vector{v}}_h}\Vert_{0,\widehat{K}}.
    \end{equation*}
    Further, from Lemma \ref{lemma:scalingHcurl}, we can obtain that
    \begin{equation}\label{eq:discreteL2NormSim1}
        \sum_{\xi\in\mathcal{N}_h(K)}\sum_{i=1}^{n_\xi}\vert M_{\xi}^i(\Vector{v}_h)\vert^2\sim \Vert{\widehat{\Vector{v}}_h}\Vert_{0,\widehat{K}}^2 \lesssim h^{2-d} \Vert{\Vector{v}_h}\Vert_{0,K}^2.
    \end{equation}
    On the other hand, it can be derived from Lemma \ref{lemma:scalingHcurl} again that
    \begin{equation}\label{eq:discreteL2NormSim2}
        \Vert{\Vector{v}_h}\Vert_{0,K}^2\lesssim h^{d-2}\Vert{\widehat{\Vector{v}}_h}\Vert_{0,\widehat{K}}^2\sim h^{d-2}\vertiii{\widehat{\Vector{v}}_h}\vertiii^2 = h^{d-2}\sum_{\xi\in\mathcal{N}_h(K)}\sum_{i=1}^{n_\xi}\vert M_{\xi}^i(\Vector{v}_h)\vert^2.
    \end{equation}
    \eqref{eq:discreteL2NormSim1} and \eqref{eq:discreteL2NormSim2} finally lead to
    \begin{align*}
        \Vert{\Vector{v}_h}\Vert_{0,h,\Omega}^2
        &= h^{d-2}\sum_{\xi\in \mathcal{N}_h(\Omega)}\sum_{i=1}^{n_\xi}\vert M_{\xi}^i(\Vector{v}_h)\vert^2 \cr
        &\sim h^{d-2}\sum_{K\in\mathcal{T}_h}\sum_{\xi\in\mathcal{N}_h(K)}\sum_{i=1}^{n_\xi}\vert M_{\xi}^i(\Vector{v}_h)\vert^2\\
        &\sim\sum_{K\in\mathcal{T}_h}\Vert{\Vector{v}_h}\Vert_{0,K}^2=\Vert{\Vector{v}_h}\Vert_{0,\Omega}^2.
    \end{align*}
\end{proof}

\bibliographystyle{siamplain}

\end{document}